\documentclass[11pt]{article}

\usepackage{amssymb}
\usepackage{amsfonts}
\usepackage{amsmath}
\usepackage{amsthm}
\usepackage{tikz}
\usepackage{float}
\usepackage{amsmath,amssymb,amsthm,amsfonts,amsbsy, enumerate}
\usepackage{tikz}
\usepackage{float}
\usepackage[dvips]{epsfig}
\usepackage{graphicx}
\usepackage{url}
\usepackage{hyperref}
\hypersetup{colorlinks=true,pdftitle="",pdftex}

\newcommand{\bc}{\begin{center}}
\newcommand{\ec}{\end{center}}
\newcommand{\bt}{\begin{tabular}}
\newcommand{\et}{\end{tabular}} 
\newcommand{\bea}{\begin{eqnarray}}
\newcommand{\eea}{\end{eqnarray}}
\newcommand{\bean}{\begin{eqnarray*}}
\newcommand{\eean}{\end{eqnarray*}}

\newcommand{\ba}{\begin{array}}
\newcommand{\ea}{\end{array}}

\def\be{\begin{eqnarray}}
\def\ee{\end{eqnarray}}
\def\ben{\begin{eqnarray*}}
\def\een{\end{eqnarray*}}

\newcommand{\RL}{{\mathbb R}}

\newcommand{\Nat}{\mathbb{N}}

\def\elabel#1{\label{e:#1}}

\def\sq{$\Box$}

\def\qed{\ifmmode\sq\else{\unskip\nobreak\hfil
\penalty50\hskip1em\null\nobreak\hfil\sq
\parfillskip=0pt\finalhyphendemerits=0\endgraf}\fi\par\medbreak}

\newsavebox{\junk}
\savebox{\junk}[1.6mm]{\hbox{$|\!|\!|$}}

\def\til={{\widetilde =}}

 \def\eq#1/{(\ref{#1})}

\def\eq#1/{(\ref{e:#1})}

\newcommand{\beqn}[1]{\notes{#1}%
\begin{eqnarray} \elabel{#1}}

\newcommand{\eeqn}{\end{eqnarray} }

\newcommand{\beq}[1]{\notes{#1}%
\begin{equation}\elabel{#1}}

\newcommand{\eeq}{\end{equation}} 

\def\bdes{\begin{description}}
\def\edes{\end{description}}

\def\notes#1{}

\theoremstyle{plain} 
\newtheorem{thm}{Theorem}[section]
\newtheorem{prop}[thm]{Proposition}
\newtheorem{cor}[thm]{Corollary}
\newtheorem{lem}[thm]{Lemma}
\newtheorem{defn}[thm]{Definition}

\theoremstyle{remark} 
\newtheorem{rmk}[thm]{Remark}

\begin{document}

\title{The norm of the Fourier transform on compact or discrete abelian groups}
\author{Mokshay Madiman and Peng Xu\thanks{Both authors are with the Department of Mathematical Sciences, University of Delaware.
This work was supported in part by the U.S. National Science Foundation through grants CCF-1346564 and DMS-1409504 (CAREER).
Email: {\tt xpeng@udel.edu, madiman@udel.edu}}}
\maketitle

\begin{abstract}
We calculate the norm of the Fourier operator from $L^p(X)$ to $L^q(\hat{X})$ when $X$ is an infinite locally compact abelian group that is, furthermore, compact or discrete. This subsumes the sharp Hausdorff-Young inequality on such groups. In particular, we identify the region in $(p,q)$-space where the norm is infinite, generalizing a result of Fournier, and setting up a contrast with the case of finite abelian groups, where the norm was determined by Gilbert and Rzeszotnik. As an application, uncertainty principles on such groups expressed in terms of R\'enyi entropies are discussed.
\end{abstract}


\section{Introduction}


The determination of the best constants in important inequalities of harmonic analysis, or equivalently the determination of the 
norm of important operators of harmonic analysis on appropriate spaces, has been an area of persistent investigation of decades. 
For instance, the norm of various variants of the Hardy-Littlewood maximal operator has attracted much attention. Indeed, 
the question of what this norm is for the centered Hardy-Littlewood maximal operator on $L^p(\mathbb{R}^d)$ for $p>1$ 
remains open even for dimension 1, even though Stein \cite{Ste82} proved that these norms are uniformly bounded as the 
dimension grows. As for the operator corresponding to the centered weak type (1,1) inequalities whose domain is $L^1(\mathbb{R}^d)$, 
the question of Stein and Stromberg \cite{SS83} about whether the norms are uniformly bounded in dimension is still unanswered, 
although Melas \cite{Mel03} determined the norm of this operator in dimension 1 in a culmination of years of effort by several authors. 
In a different line of investigation, the norm of the uncentered Hardy-Littlewood maximal operator was determined by Bernal \cite{Ber89} 
for the weak (1,1) transform in 1 dimension, and by Grafakos and Montgomery-Smith \cite{GM97} on $L^p(\mathbb{R}^d)$.

Arguably the most basic operator in harmonic analysis is the Fourier operator- the operator that takes a function to its Fourier transform, 
and it also is the common element in both real-variable harmonic analysis and the abstract theory of harmonic analysis on locally compact groups. 
Because of the nature of the Fourier transform, it is natural to ask not just about the boundedness of the Fourier operator on a given $L^p$ space, 
but for its boundedness as an operator from $L^p$ to $L^q$. The boundedness of the Fourier operator from $L^p(\mathbb{R}^d)$ to 
$L^{p'}(\mathbb{R}^d)$ with $p'$ the dual index to $p$ and $p\in (1,2]$, is precisely the content of the Hausdorff-Young inequality, and 
Beckner \cite{Bec75} obtained the norm of this operator in a celebrated paper. The more general question of the $(p,q)$-norm of the 
Fourier operator on $\mathbb{R}^d$ can be deduced from general results of Lieb \cite{Lie90} about best constants in a wide class of inequalities; 
indeed, it turns out that for $(p,q)$ either in the set $\{1< p\le 2, 1<q<\infty\}$ or in the set $\{1<p<\infty, 2\le q<\infty\}$, 
the $(p,q)$-norm is infinite unless $q= p'€™$. Somewhat surprisingly, however, until recently it does not appear that the norm of the 
Fourier operator had been explored in the abstract group setting, even though the question continues to make perfect sense there 
(even though maximal functions do not). The only work we are aware of beyond $\mathbb{R}^d$ is that of Gilbert and Rzeszotnik \cite{GR10}, 
who settled the determination of the $(p,q)$-norm of the Fourier operator for arbitrary finite abelian groups.

Our goal in this paper is to determine the $(p,q)$-norm of the Fourier operator for the larger class of compact or discrete abelian groups. 
Part of the motivation for this comes from the fact that the extremal functions obtained by Gilbert and Rzeszotnik for the case of 
finite abelian groups do not easily extend to cases where one does not have a discrete topology (e.g., the natural analogue of some of 
their extremal functions would be something akin to an Euler delta function, but these fail to be in the nice function spaces of interest 
and also are not straightforward to define and develop in the group setting where one cannot work as concretely as in $\mathbb{R}^d$ 
or finite groups). As a consequence, our proof techniques, while building on those of Gilbert and Rzeszotnik, are necessarily more involved, 
and for example rely on explicit constructions even in cases where Gilbert and Rzeszotnik were able to make existence arguments suffice.

For finite abelian groups $X$, the norm of the Fourier transform as an operator from $L^p(X)$ to $L^q(\hat{X})$ is clearly finite for all positive $p$ and $q$, and this
number is computed as earlier mentioned in \cite{GR10}. Perhaps the main surprise when dealing with the more general situations
of compact or discrete abelian groups is that, {\it provided the group is not finite}, there is a region in which the $(p,q)$-norm of the Fourier transform 
is infinite. More precisely, if $X$ is compact and infinite, the Fourier transform is only finite in the region $R_1$ described in Theorem~\ref{ThmXCom}
while if $X$ is discrete and infinite, the Fourier transform is only finite in the region $R'_2$ described in Theorem~\ref{ThmXDis}.
The regions $R_i$ that appear in Theorem~\ref{ThmXCom} are almost the same as the regions $R'_i$ that appear in Theorem~\ref{ThmXDis}
-- they differ only on their boundaries. In other words, denoting the interior of a set $A$ by $A^\circ$, we have
that $R_i^\circ=R_i^{'\circ}$ for $i=1, 2, 3$; these regions are shown in Figure~\ref{figure:2}.
A subset of our main results was developed earlier by Fournier \cite{Fou73}-- specifically, he showed that the norm of the Fourier 
transform is infinity in the region $\{1\le p\le 2, \frac{1}{p}+\frac{1}{q}>1\}$ for $X$ compact, and in the region $\{1\le p\le 2, \frac{1}{p}+\frac{1}{q}<1\}$ for $X$ discrete.
We emphasize that our results cover all pairs $(p,q)$ in the positive quadrant of the extended plane (appropriately interpreted
when $p$ or $q$ are less than 1 and we are not dealing with a Banach space).
In this sense, the range of values we consider in this paper is more general than the range considered by \cite{GR10} for finite abelian groups, where
only the usual (norm) case of $p\geq 1, q\geq 1$ is considered.

Our proofs of the compact and discrete cases are distinct. It is conceivable that one may be able to use an argument 
based on duality to derive one from the other; however we found it more convenient to develop them separately,
especially because, for the question to be well defined, 
we need to sometimes work with a subspace of $L^p(X)$ rather than the whole space (as explained in the second paragraph
of Section~\ref{sec:prelim}).
%
%
Let us note in passing that Theorems~\ref{ThmXCom} and \ref{ThmXDis} of course contain the 
Hausdorff-Young inequality (see, e.g., \cite{Bec75}) for compact or discrete LCA groups 
with the sharp constant $C_{p,q}=1$ for $\frac{1}{p}+\frac{1}{q}=1$. 

As is well known, the Hausdorff-Young inequality for $\RL^n$, used together with the associated sharp constant, yields on differentiation
a sharp uncertainty principle for the Fourier transform on $\RL^n$ expressed in terms of entropy; 
in particular, the Heisenberg uncertainty principle for position and momentum 
observables can be extracted as a consequence. For some parameter ranges, our results can analogously be interpreted in terms of certain
uncertainty principles for the Fourier transform on LCA groups.


%
%
\begin{figure}[H]
	\centering
	\begin{tikzpicture}[scale=1.5,dot/.style={draw,circle,inner sep=2pt,fill},lab/.style={outer sep=2pt,font=\large}]
		\draw [thick,->](0,0)--(0,5);
		\draw [thick,->](0,0)--(6,0);
		\draw (2,2)--(4,0);
		\draw (2,2)--(2,5);
		\draw (2,2)--(0,2);
		\node[dot] at (0,0) {} node[lab] at (0,0) [below left] {$0$};
		\node[dot] at (2,0) {} node[lab] at (2,0) [below] {$\frac{1}{2}$};
		\node[dot] at (4,0) {} node[lab] at (4,0) [below] {$1$} node[lab] at (6,0) [above right] {$\frac{1}{p}$};
		\node[dot] at (0,2) {} node[lab] at (0,2) [left] {$\frac{1}{2}$};
		\node[lab] at (0,5) [above right] {$\frac{1}{q}$};
		\node[dot] at (2,2) {};
		\node[lab] at (1,3.5) {$R_3^\circ=R_3^{'\circ}$};
		
		\node[lab] at (4,2.1) {$R_2^\circ=R_2^{'\circ}$};
		
		\node[lab] at (1.6,1) {$R_1^\circ=R_1^{'\circ}$};

	\end{tikzpicture}
	\caption{The three regions demarcated by the lines in the figure are the interiors of the regions that
	show up in Theorems~\ref{ThmXCom} and \ref{ThmXDis}. For an infinite LCA group $X$,
	the $(p,q)$-norm of the Fourier transform is finite precisely on $R_1$ when $X$ is compact
	and precisely on $R'_2$ when $X$ is discrete; in particular, it is infinite on the third region in either case. 
	}
	\label{figure:2}
\end{figure}
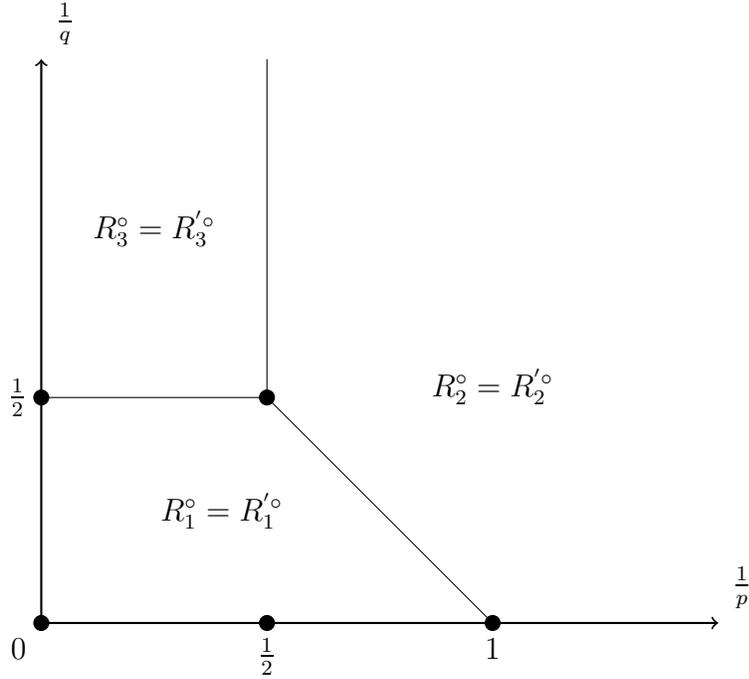

%
%
%
%

Section~\ref{sec:prelim} contains some preliminary material on abstract harmonic analysis as
well as probability that we will need to set notation and for our proofs.
Section~\ref{sec:compact} states and proves the main result for compact abelian groups, while
Section~\ref{sec:discrete} states and proves the main result for discrete abelian groups.

There is a large literature on uncertainty principles in both Euclidean and abstract settings. 
There are many different ways to express the intuition
that both a function and its Fourier transform cannot be simultaneously too concentrated. For the standard
setting of the Euclidean spaces $\RL^n$, these include the following formulations:
\begin{enumerate}
\item Hardy-type uncertainty principles: The simplest forms of these assert that if both $f$ and $\hat{f}$ are non-zero and sub-gaussian,
then the constants in the bounding Gaussian functions must be constrained. See, e.g., \cite{Har33, CP84, Hor91, BD06}.
\item Heisenberg-type uncertainty principles: These assert that the product of the variances of conjugate densities must be 
bounded from below by a positive constant. See, e.g., \cite{Ken27, Wey50:book, Rob29, Kra67, deB67, Chi76}.
\item Uncertainty principles \`a la Amrein-Berthier or Logvinenko-Sereda: These assert that if the energy of $f$ is largely concentrated in a ``small'' set $E$
and that of $\hat{f}$ is largely concentrated in a ``small'' set $F$, then $f$ itself must have small energy ($L^2$-norm). See, e.g., \cite{LS74, HJ95, Pan95, Pan98, Kov01} for 
statements where ``small'' means compact, and \cite{AB77, Naz96, Jam07} for statements where ``small'' means of finite Lebesgue measure. 
\item Entropic uncertainty principles: These assert that the sum of entropies of a density and its conjugate density
is bounded from below by a constant. See \cite{Hir57, Sta59, Bec75, DCT91} for various statements of this type involving Shannon entropy, 
\cite{DHPO05} for an exploration of extremals, and \cite{ZPV08} for extensions from Shannon entropy to the 
more general R\'enyi entropies (the latter two papers also discuss finite cyclic groups and the integers).
\end{enumerate}

Efforts have been made to find more general expressions of each of these forms of uncertainty principles
that can apply to settings more general than Euclidean spaces or to more general objects than the Fourier transform; 
a detailed discussion of the literature is beyond the scope of this paper but the references given above together 
with the surveys \cite{Fef83, FS97:1, HJ94:book, Tha04:book, Sen14} provide a starting point. 

In Section~\ref{sec:up}, we discuss the implications of our results for entropic uncertainty principles on compact or discrete abelian groups
(expressed using the general class of R\'enyi entropies), and compare with results that can be deduced directly from the Hausdorff-Young inequality.
Once again it is particularly interesting to note that there is a region of $(p,q)$-space in which the natural weighted uncertainty inequality
involving R\'enyi entropies of orders $p$ and $q$ does not hold.

Other papers that explore uncertainty principles in the setting of LCA groups include 
\cite{MS73, Smi90, Hog93} who study Amrein-Berthier-type phenomena, and \cite{OP04, Prz04} who study entropic uncertainty principles
expressed in terms of Shannon entropy. Note that it is unclear how to directly extend Heisenberg-type inequalities
involving the variance from $\RL^n$ to LCA groups where the moments of a group-valued random variable
are not well defined; this is one advantage of entropic uncertainty principles, which imply Heisenberg-type uncertainty
principles in the Euclidean case, but can be freely formulated for groups since the entropy is defined only in terms
of the density and not moments.

We mention in passing that it is, of course, of great interest to obtain norms of Fourier transforms, and sharp constants in uncertainty principles, 
for {\it nonabelian} locally compact groups (starting with the most common Lie groups) or other general structures such as metric measure spaces,
but this is a much harder project that we do not address at all. For the current state of knowledge on such questions,
the reader may  consult \cite{Chr04, Kan07, PT07, DT15, MM15} and references therein.

\section{Preliminaries}
\label{sec:prelim}

Since this section contains basic material on the Fourier transform on LCA groups that can be found in
textbooks (see, e.g., \cite{Rud62:book}), we do not include most of the proofs.

Let $X$ be a LCA group associated with a Haar measure $\alpha$. Let $\hat{X}:=\{\gamma:~X\rightarrow\mathbb{T}\}$ be its dual group (i.e. the set of continuous group homomorphism from $X$ to the unit circle $\mathbb{T}$ in complex plane). If $X$ is compact then $\hat{X}$ is discrete and the Haar measure $\hat{\alpha}$ on $\hat{X}$ is the counting measure. If $X$ is discrete then $\hat{X}$ is compact. Consider $1\le p,q\le \infty$ and the spaces $L^p(X)$ and $L^q(\hat{X})$ with the corresponding Haar measures. For an integrable function $f:X\rightarrow\mathbb{C}$, we define its Fourier transform:
\begin{eqnarray}
\hat{f}(\gamma)=\int_Xf(x)\gamma(-x)\alpha(dx)\label{FT}
\end{eqnarray}
Note that the Fourier transform is a linear transform and can be extended, in a unique manner, to an isometry of $L^2(X)$ onto $L^2(\hat{X})$. 
If $X$ is a compact LCA group, define the norm of Fourier transform from $L^1(X)\cap L^p(X)$ to $L^q(\hat{X})$ for $0< p\le \infty$ by
\be\label{FTNorm}
C_{p,q}:=\sup_{\|f\|_p=1}\|\hat{f}\|_q
\ee
Note that if $p\ge 1$, then $L^1(X)\cap L^p(X)=L^p(X)$; otherwise if $0<p<1$, then $L^1(X)\cap L^p(X)=L^1(X)$.
The reason why we define Fourier transform for functions on $L^1(X)\cap L^p(X)$ instead of $L^p(X)$ for $0<p<1$ is that there is no direct definition for Fourier transform of functions on $L^p(X)$ that are not integrable. Similarly, if $X$ is a discrete LCA group, define the norm of Fourier transform from $L^2(X)\cap L^p(X)$ to $L^q(\hat{X})$ for $0< p\le \infty$  by
\be\label{FTNormdisc}
C_{p,q}:=\sup_{\|f\|_p=1}\|\hat{f}\|_q
\ee
Note that if $1\le p\le 2$, then $L^2(X)\cap L^p(X)=L^p(X)$, otherwise if $p>2$, then $L^2(X)\cap L^p(X)=L^2(X)$. The goal of this paper is to explore the explicit value of the norm of Fourier operator for arbitrary $0< p,q\le\infty$.

\begin{defn}\label{InP}
Suppose that $X$ is an LCA group with a Haar measure $\alpha$, then for any $f$ and $g$ in $L^2(X)$, define the inner product on $L^2(X)$: 
\be\label{InnPro}
\langle f,g\rangle:=\int_X f(x)\overline{g(x)}\alpha(dx)
\ee
\end{defn}

\begin{prop}\label{Ortho}
Suppose that $(X,+)$ is a compact abelian group, then the elements of $(\hat{X},\cdot)$ form an orthogonal basis in $L^2(X)$. We call this basis the frequency basis of $X$. If $(X,\cdot)$ is a discrete abelian group, then all delta functions 
\begin{eqnarray}\label{TimBas}
\mathbb{I}_{x_0}(x):=\begin{cases}
1 &\mbox{if~}x=x_0\\
0 &\mbox{otherwise}
\end{cases}
\end{eqnarray}
form an orthonormal basis of $L^2(X)$, we call this basis the time basis of $L^2(X)$. 
\end{prop}


\begin{defn}\label{PosDef}
Let $(X,+)$ be an LCA group, a function $f$ on $X$ is said to be positive-definite if
\ben
\sum_{i,j=1}^nc_i\overline{c_j}f(x_i-x_j)\ge 0
\een
holds for every choice of $x_i\in X$, $c_i\in \mathbb{C}$, $i,j\in\{1,\cdots,n\}$ and $n\in\mathbb{N}$. Denote $B(X)$ the set of positive-definite functions on $X$.
\end{defn}

\begin{thm}\label{InvFT}
Let $f\in B(X)\cap L^1(X)$, then $\hat{f}\in L^1(\hat{X})$ and,
the Haar measure $\hat{\alpha}$ on $\hat{X}$ can be normalized such that the following inversion formula holds for all $x\in X$:
\begin{eqnarray}\label{Invft}
f(x)=\int_{\hat{X}}\hat{f}(\gamma)\gamma(x)\hat{\alpha}(d\gamma)
\end{eqnarray}
\end{thm}


\begin{prop}\label{HarMeas}
Suppose that $X$ is a compact LCA group with a Haar measure $\alpha$. Suppose $\hat{X}$ is the dual group of $X$ with the Haar measure $\hat{\alpha}$ normalized in Theorem \ref{InvFT}, then $\hat{\alpha}$ is a counting measure on $\hat{X}$ such that for every $y\in \hat{X}$, $\hat{\alpha}(y)=\alpha(X)^{-1}$. On the other hand, if $X$ is a discrete LCA group with a Haar measure $\alpha$. Suppose $\hat{X}$ is the compact dual group of $X$ with the Haar measure $\hat{\alpha}$ normalized in Theorem \ref{InvFT}, then $\alpha$ is a counting measure such that for every $x\in X$, $\alpha(x)=\hat{\alpha}(\hat{X})^{-1}$.
\end{prop}

\begin{proof}
If $X$ is compact, then for every element $\gamma\in\hat{X}$, $\gamma$ is a continuous homeomorphism from $X$ to $\mathbb{T}$. Then the Fourier transform, $\hat{\gamma}$, of $\gamma$ is a delta function on $\hat{X}$
\begin{eqnarray}\label{Dta}
\hat{\gamma}(y)=\begin{cases}
\alpha(X) &\mbox{if}~y=\gamma\\
0 &\mbox{if}~y\neq\gamma
\end{cases}
\end{eqnarray}
by the definition of Fourier transform \ref{FT} and Proposition  \ref{Ortho}. On the other hand, it is easy to verify that $\gamma\in B(X)$, therefore by inversion Fourier transform (\ref{Invft}), we have
\begin{eqnarray}
\gamma(x)=\int_{\hat{X}}\hat{\gamma}(y)y(x)\hat{\alpha}(dy)=\alpha(X)\hat{\alpha}(\gamma)\gamma(x)
\end{eqnarray}
therefore we have $\hat{\alpha}(\gamma)=\big(\alpha(X)\big)^{-1}$ for every $\gamma\in\hat{X}$, which provides the conclusion. On the other hand, if $X$ is discrete, then it suffices to prove that $\alpha(e)=\hat{\alpha}(\hat{X})^{-1}$ ($e$ is the identity in $X$) by the fact that $\alpha$ is a Haar measure and that $X$ is discrete. Define a function $f$ on $X$ by
\begin{eqnarray}
f:=\mathbb{I}_{e}(x):=\begin{cases}
1 &\mbox{if~} x=e\\
0 &\mbox{otherwise}
\end{cases}
\end{eqnarray}
We claim that $f\in B(X)$. In fact, for every choice of $x_i\in X$, $c_i\in\mathbb{C}$, $i,j\in\{1,\cdots,n\}$ and $n\in\mathbb{N}$, we divide $\{1,\cdots,n\}$ into several classes through the values of $\{x_i\}$ (since $x_i$'s are not necessarily distinct), i.e. we divide $\{1,\cdots,n\}$ into $m$ classes with $m\le n$: $A_1,A_2,\cdots,A_m$ such that if $i,j\in A_k$ with $k\le m$, then $x_i=x_j$, if else $i\in A_{k_1}$ and $j\in A_{k_2}$ with $k_1\neq k_2$, then $x_i\neq x_j$. We claim that
\be\label{Eq:class}
\sum_{i,j=1}^nc_i\overline{c_j}f(x_i-x_j)=\sum_{k=1}^m\sum_{i,j\in A_k}c_i\overline{c_j}f(x_i-x_j)
\ee
To see this, note that if $i$ and $j$ come from different $A_k$'s, then $x_i\neq x_j$, therefore by the definition of $f$, we have $f(x_i-x_j)\equiv 0$, which provides (\ref{Eq:class}). Thus it suffices to prove that
\begin{eqnarray}
\sum_{i,j\in A_k}c_i\overline{c_j}f(x_i-x_j)\ge 0
\end{eqnarray}
holds for every $k\in \{1,\cdots,m\}$. We have
\begin{eqnarray*}
\sum_{i,j\in A_k}c_i\overline{c_j}f(x_i-x_j)=\sum_{i,j\in A_k}c_i\overline{c_j}=\Big(\sum_{i\in A_k}c_i\Big)\overline{\Big(\sum_{i\in A_k}c_i\Big)}\ge 0
\end{eqnarray*}
which provides the claim that $f\in B(X)$. Also it is easy to verify that $f\in L^1(X)$, thus the inverse Fourier transform holds by Theorem \ref{InvFT}. We have, for every $\gamma\in\hat{X}$, $\hat{f}(\gamma)\equiv \alpha(e)$. By \ref{Invft}, we have 
\ben
1=f(e)=\int_{\hat{X}}\hat{f}(\gamma)\gamma(e)\hat{\alpha}(d\gamma)=\alpha(e)\hat{\alpha}(\hat{X})
\een
which provides Proposition \ref{HarMeas}. 
\end{proof}

\begin{prop}\label{Unitary}
Suppose that $X$ is either a compact or a discrete LCA group with a Haar measure $\alpha$, then under the normalization of $\hat{\alpha}$ by Proposition \ref{HarMeas}, we have that the Parseval's identity holds for every function $f\in L^2(X)$:
\be\label{Uni}
\|f\|_2=\|\hat{f}\|_2
\ee
\end{prop}


\begin{thm}\label{thm:fouriertwice}
Suppose that $X$ is a compact LCA group with a Haar measure $\alpha$. Let $\hat{X}$ be its dual group with the Haar measure $\hat{\alpha}$ normalized so that the Plancherel identity is valid. Then if a function $f$ whose Fourier transform $\hat{f}\in L^{p}(\hat{X})$ for some $p\in (0,2]$, then $\hat{\hat{f}}(x)=f(-x)$. On the other hand, if $X$ is a discrete LCA group and $f$ a function whose Fourier transform $\hat{f}\in L^p(\hat{X})$ for some $p\in [2,\infty]$, then $\hat{\hat{f}}(x)=f(-x)$.
\end{thm}

The crucial idea to calculate all the finite operator norms in this paper, as in the paper of Gilbert and Rzeszotnik \cite{GR10}, 
is the Riesz-Thorin theorem. 

\begin{thm}\label{RiesT}
\textbf{(Riesz-Thorin Convexity Theorem)} Define $K(1/p,1/q):=C_{p,q}$, where $C_{p,q}$ is the norm of Fourier operator defined in (\ref{FTNorm}), then $\log K(x,y)$ is a convex function on $[0,\infty)\times[0,1]$.
\end{thm}


As in \cite[Corollary 2.2]{GR10}, one has the following consequence of the Riesz-Thorin theorem. 

\begin{cor}\label{corRT}
If $f$ is an affine function and $\log K(p)\le f(p)$ for all $p$ in a finite set $P\in [0,\infty)\times[0,1]$ then $\log K(p)\le f(p)$ for all $p\in \mbox{hull}(P)$, where 
\begin{eqnarray*}
&&P:=\{p_1,p_2,\cdots,p_l\}\\
&&\mbox{hull}(P):=\Big\{\sum_{p_i\in P} \lambda_ip_i:\lambda_i\ge 0~\mbox{for~all}~i~\mbox{and}~\sum_{i}\lambda_i=1\Big\}
\end{eqnarray*}
\end{cor}


We need some facts about lacunary series (see, e.g., \cite{Zyg59:book}), 
which are series in which the terms that differ from zero are ``very sparse". 

\begin{defn}\label{LacunarySeris}
A lacunary trigonometric series is a series that can be written in the form
\begin{eqnarray}
\sum_{k=1}^\infty (a_k\cos n_kx+b_k\sin n_kx)=:\sum_{k=1}^\infty A_{n_k}(x) ,
\end{eqnarray}
where the sequence $(n_k:k\in\Nat)$ satisfies 
\ben
\frac{n_{k+1}}{n_k} \ge \Lambda>1  
\een
for each $k\in\Nat$.
\end{defn}



\begin{thm}\label{Thm:DependCLT}\cite[Vol. 2, p. 264]{Zyg59:book} 
Consider a lacunary trigonometric series 
\be\label{eq:LacunarySeries}
\sum_{k=1}^\infty(a_k\cos n_kx+b_k\sin n_kx)=\sum r_k\cos (n_kx+x_k)=:\sum A_{n_k}(x)
\ee
where $n_{k+1}/n_k\ge \Lambda>1$ and $r_k\ge 0$ for all $k$. Write
\begin{eqnarray*}
S_k(x)&:=&\sum_{j=1}^k(a_j\cos n_jx+b_j\sin n_jx)\\
A_k &:=&\left(\frac{1}{2}\sum_{j=1}^k(a_j^2+b_j^2)\right)^{1/2}
\end{eqnarray*}
Suppose that $A_k$ and $r_k$ satisfies 
\begin{eqnarray}\label{eq:ExtraLacunaryCondition}
(i)~A_k\rightarrow\infty,~~(ii)~r_k/A_k\rightarrow 0
\end{eqnarray}
Under the hypothesis (\ref{eq:ExtraLacunaryCondition}), the functions $S_k(x)/A_k$ are asymptotically Gaussian distributed on each set $E\subset(0,2\pi)$ of positive measure, that is, 
\begin{eqnarray}\label{eq:LacunaryCLT}
\frac{\lambda\left\{x\in E:~S_k/A_k\ge y\right\}}{\lambda(E)}\rightarrow \frac{1}{2\pi} \int_y^{\infty}e^{-x^2/2}dx
\end{eqnarray}
where $\lambda$ is the Lebesgue measure on $[0,2\pi]$.
\end{thm}

Finally we need the Lyapunov Central Limit Theorem.

\begin{thm}\label{LyapCLT}
Suppose $\{Y_1, Y_2, \cdots\}$ is a sequence of independent random variables, each with finite expected value $\mu_k$ and variance $\sigma_k^2$. Define
$$
s_n^2:=\sum_{k=1}^n\sigma_k^2
$$
If for some $\delta>0$, the Lyapunov condition
\begin{eqnarray}\label{LyapCdt}
\lim_{n\rightarrow\infty}\frac{1}{s_n^{2+\delta}}\sum_{k=1}^n\mathbb{E}(|Y_k-\mu_k|^{2+\delta})=0
\end{eqnarray}
is satisfied, then as $n\rightarrow\infty$,
\begin{eqnarray}\label{LyapConcl}
\frac{1}{s_n}\sum_{k=1}^n(Y_k-\mu_k)\xrightarrow{d}\mathcal{N}(0,1)
\end{eqnarray}
where $\mathcal{N}(0,1)$ is the standard normal distribution.
\end{thm}

\section{The compact case}
\label{sec:compact}

From now on, we always assume that our LCA groups are non-finite. 
From now on, we define the group operator as ``$+$" or ``$\cdot$" for when the group is compact or discrete respectively.

\begin{thm}\label{ThmXCom}
If $X$ is a compact non-finite LCA group, we consider three regions as in Figure. 1:
\begin{eqnarray}
&&R_1:=\Big\{\Big(\frac{1}{p},\frac{1}{q}\Big)\in [0,\infty)^2:\frac{1}{p}+\frac{1}{q}\le 1, \frac{1}{q}\le \frac{1}{2}\Big\}\\
&&R_2:=\Big\{\Big(\frac{1}{p},\frac{1}{q}\Big)\in [0,\infty)^2:\frac{1}{p}+\frac{1}{q}> 1, \frac{1}{p}> \frac{1}{2}\Big\}\\
&&R_3:=\Big\{\Big(\frac{1}{p},\frac{1}{q}\Big)\in [0,\infty)^2:\frac{1}{p}\le \frac{1}{2}, \frac{1}{q}>\frac{1}{2}\Big\}
\end{eqnarray}
Then the norm of the Fourier operator from $L^p(X)$ to $L^q(\hat{X})$ satisfies 
\begin{eqnarray*}
C_{p,q}=\begin{cases}
\alpha(X)^{1-1/p-1/q} & \mbox{if}~\Big(\frac{1}{p},\frac{1}{q}\Big)\in R_1\\
\infty & \mbox{if}~\Big(\frac{1}{p},\frac{1}{q}\Big)\in R_2\cup R_3
\end{cases}
\end{eqnarray*}
\end{thm}


\begin{prop}\label{PropR1XCom}
If $(1/p,1/q)\in R_1$, then $C_{p,q}= \alpha(X)^{1-1/p-1/q}$.
\end{prop}

\begin{proof} [Proof of Proposition \ref{PropR1XCom}]
We firstly prove that $C_{p,q}\le  \alpha(X)^{1-1/p-1/q}$. Since the region $R_1$ is convex and 
$$
R_1=\mbox{hull}\Big((0,0),(0,1),(1/2,1/2),(0,1/2)\Big)
$$
and $\log \alpha(X)^{1-1/p-1/q}$ is affine for $(1/p,1/q)$, it suffices to check that $C_{p,q}\le \alpha(X)^{1-1/p-1/q}$ holds at $(0,0),(1,0),(1/2,1/2),(0,1/2)$ by Corollary \ref{corRT}. \\

At the point $(0,0)$, we have 
\begin{eqnarray}
\|\hat{f}\|_\infty\le \|f\|_1\le \alpha(X)\|f\|_\infty
\end{eqnarray}
so $C_{\infty,\infty}\le \alpha(X)$. \\

At the point $(1,0)$, we have 
\begin{eqnarray}
\|\hat{f}\|_\infty\le \|f\|_1
\end{eqnarray}
so $C_{1,\infty}\le 1$.\\

At the point $(1/2,1/2)$, we have that $C_{2,2}=1$ by the fact that Fourier transform is unitary. \\

At the point $(0,1/2)$, we have 
\begin{eqnarray}
\|\hat{f}\|_2=\|f\|_2=\Big(\int_X|f|^2\alpha(dx)\Big)^{1/2}\le \alpha(X)^{1/2}\|f\|_\infty
\end{eqnarray}
so $C_{\infty,2}\le \alpha(X)^{1/2}$, which provides $C_{p,q}\le  \alpha(X)^{1-1/p-1/q}$. \\

Further, the upper bound $\alpha(X)^{1-1/p-1/q}$ can be attained by any frequency basis $\gamma\in \hat{X}$. In fact, $\|\gamma(x)\|_p=\alpha(X)^{1/p}$. By (\ref{Dta}) and Proposition \ref{HarMeas}, $\|\hat{\gamma}\|_q=\alpha(X)^{1-1/q}$, this yields the fact $C_{p,q}\ge \alpha(X)^{1-1/p-1/q}$, which ends the proof. 
\end{proof}

\begin{prop}\label{PropR2XCom}
If $(1/p,1/q)\in R_2$, then $C_{p,q}= \infty$.
\end{prop}

We will actually prove a stronger condition:  If $(1/p,1/q)\in [0,\infty)^2$ with $1/p+1/q>1$, then $C_{p,q}=\infty$. Since $X$ is compact, the Haar measure $\alpha$ on $X$ is finite. Without loss of generality, we assume that $\alpha$ is a probability measure $\mathbb{P}$. On the other hand, since $\hat{X}$ is discrete, we have three possible cases:
\begin{itemize}
\item Case 1. $\hat{X}$ has an element with order infinity.
\item Case 2. Every element of $\hat{X}$ has a finite order, and the order set is not bounded.
\item Case 3. Every element of $\hat{X}$ has a finite order, and the order set is uniformly bounded.
\end{itemize}

We will prove Proposition \ref{PropR2XCom} through these three cases after the following lemmas that will be used repeatedly:

\begin{lem}\label{Rangeofg}
Suppose that $X$ is a compact LCA group, then for every element $g\in \hat{X}$, we have that the image $g(X)$ is a closed subgroup of $\mathbb{T}$
\begin{itemize}
\item If the order of $g$ is infinity, then $g(X)$ is a dense subgroup of $\mathbb{T}$. 
\item If the order of $g$ is $m<\infty$, then $g(X)=[e^{2\pi i/m}]:=\{e^{2\pi ik/m},~k=0,1,2,\cdots,m-1\}$.
\end{itemize}
\end{lem}

\begin{proof} [Proof of Lemma \ref{Rangeofg}] 
We firstly prove that $g(X)$ is closed. If $g(X)$ is a finite set, then $g(X)$ is clearly closed. If $g(X)$ is not a finite set, then assume a sequence $\{c_n\}_{n=1}^\infty\subset g(X)$ and $c_n\rightarrow c_0\in\mathbb{T}$, then we will prove that $c_0\in g(X)$. In fact, construct a sequence, $\{x_n\}_{n=1}^\infty$, in $X$ by setting $x_n\in g^{-1}(\{c_n\})$, then we claim that the sequence $\{x_n\}_{n=1}^\infty$ has a limit point in $X$. Suppose otherwise, then for every element $x\in X$, there exists an open neighborhood, $U_{x}$, of $x$, such that the cardinality of $U_{x}\cap \{x_n\}_{n=1}^\infty$ is either $0$ or $1$. Then all $\{x\in X, U_{x}\}$ form an open cover of $X$. Because $X$ is compact, then there exist a finite subcover of $\{U_{x}\}$ that covers $\hat{X}$. Denote this subcover $\{U_1,U_2,\cdots,U_m\}$, then the cardinality of the sequence $\{x_n\}_{n=1}^\infty=\{x_n\}_{n=1}^\infty\cap \bigcup_{j=1}^m U_j=\bigcup_{j=1}^m(\{x_n\}_{n=1}^\infty\cap U_j)$ is finite, which yields a contradiction. Therefore $\{x_n\}_{n=1}^\infty$ has a limit point in $X$. Denote $x_0$ this limit point. We claim that $g(x_0)=c_0$. Suppose if otherwise $g(\hat{x}_0)=c_0'\neq
c_0$, then there exists an open neighborhood, $V_0$, of $c_0'$ such that $V_0\cap \{c_n\}_{n=1}^\infty$ has only finitely many elements. Hence the pre-image, $g^{-1}(V_0)\cap \{\hat{x}_n\}_{n=1}^\infty$ has only finitely many elements. However, $g^{-1}(V_0)$ is also an open neighborhood of $x_0$, then by the fact that $x_0$ is a limit point of $\{x_n\}_{n=1}^\infty$, which means that $g^{-1}(V_0)\cap \{x_n\}_{n=1}^\infty$ should have infinitely many elements, which leads to a contradiction. This provides the fact that $g(X)$ is closed in $\mathbb{T}$.\\

We claim that $g(X)$ is either dense on $\mathbb{T}$ or is a finite subgroup on $\mathbb{T}$. It is easy to see that $g(X)$ is a subgroup of $\mathbb{T}$. Suppose $1\in\mathbb{T}$ is a limit point of $g(X)$, then we have a sequence $ a_n\rightarrow 0$ such that $e^{ia_n}\in g(X)$, then we have $e^{i|a_n|}\in g(X)$ and $e^{ik|a_n|}\in g(X)$ for all $k\in\mathbb{Z}$, which provides that $g(X)$ is dense in $\mathbb{T}$. Further assume $1$ is not a limit point of $g(X)$, then since $g(X)$ is closed in $\mathbb{T}$, we take the element in $g(X)$ that is closest to 1, denote this element by $e^{is}$ (without loss of generality assume $s>0$), then we claim that $g(X)=[e^{is}]$ and $s|2\pi$. To see this, if $s$ does not divide $2\pi$, then we have a remainder $s'$ with $0<s'<s$ such that $e^{is'}\in g(X)$, contradicts to the fact that $e^{is}$ is closest to 1. Similarly, if there eixsts some $t$ such that $t\notin[e^{is}]$, then there exists a $k\in \mathbb{Z}$ such that $e^{it}\in (e^{iks}, e^{i(k+1)s})$, thus we have $g(X)\ni e^{i(t-ks)}\neq e^{is}$ with $|t-ks|<s$, which leads to the same contradiction. \\

By the above argument, if $g$'s order is infinity, then $g(X)$ must be dense on $\mathbb{T}$. If $g$ has order $m$, then clearly $g(X)$ is a subgroup of $[e^{2\pi i/m}]$, then $g(X)=[e^{2\pi i/m'}]$ with $m'|m$, then $g$ has order $m'$, which provides $m'=m$ and therefore $g(X)=[e^{2\pi i/m}]$. 
\end{proof}

\begin{lem}\label{Distribofg}
Suppose $X$ is a compact LCA group and $g\in\hat{X}$. We treat $g$ as a random variable: $X\rightarrow\mathbb{T}$, we have
\begin{itemize}
\item If the order of $g$ is infinity, denote $\lambda'$ the probability measure on $\mathbb{T}$ induced by the probability distribution of the random variable $g$, i.e. for every measurable set $M\subset\mathbb{T}$, $\lambda'(M):=\mathbb{P}(x\in X:~g(x)\in M)$. Denote $\lambda$ the Lebesgue measure on $\mathbb{T}$, then $\lambda'=\lambda/2\pi$ and $\lambda(\setminus g(X))=2\pi$. 
\item If the order of $g$ is $m<\infty$, then $g: X\rightarrow[e^{2\pi i/m}]$ is a random variable with uniform distribution on $[e^{2\pi i/m}]$.
\end{itemize}
\end{lem}

\begin{proof}[Proof of Lemma \ref{Distribofg}] 
We have two claims: \\

\textbf{Claim 1. } For any $e^{ia}\in S:=g(X)$ and any measurable set, $E$, of $\mathbb{T}$, we have
\begin{eqnarray}\label{Transltinv}
\lambda'(E)=\lambda'(e^{ia}E)
\end{eqnarray}
where $\lambda'$ is the induced probability measure and $e^{ia}E:=\{e^{ia+ix}:e^{ix}\in E\}$.\\

To see this, note that if $f(x_0)=e^{ia}$
\begin{eqnarray}\label{eqvltset}
\{x_0x:g(x)\in E\}=\{y:g(y)\in e^{ia}E\}
\end{eqnarray}
In fact, for every $x_0x$ contained in the left hand side of (\ref{eqvltset}), we have
$$
g(x_0x)=g(x_0)g(x)=e^{ia}g(x)\in \{y:g(y)\in e^{ia}E\}
$$
On the other hand, for every $y$ contained on the right hand side of (\ref{eqvltset}), we have $y=x_0x$ for some $x\in X$, and 
$$
g(x)=g(x_0)^{-1}g(y)\in e^{-ia}e^{ia}E=E
$$
which provides (\ref{eqvltset}) and hence (\ref{Transltinv}) by the fact that $\mathbb{P}$ is a Haar measure. \\

If $g$ has order $m$, then it is clear to verify by Claim 1 that $g$ is a random variable with uniform distribution on $[e^{2\pi i/m}]$.\\

\textbf{Claim 2.} If $g$'s order is infinity, then the induced probability measure $\lambda'$ on $\mathbb{T}$ is a continuous measure, i.e. the distribution function on $[0,2\pi)$ is a continuous function. \\

To see this, suppose that $\lambda'$ is not a continuous measure, then by the fact that the distribution function on $[0,2\pi)$ is monotone, thus there exists countably many jumps. Assume $y_0\in \mathbb{T}$ is a jump point, thus $\lambda'(y_0)>0$ and $y_0\in g(X)$. Therefore by Claim 1, we have $\lambda'(1)=\lambda'(y_0)>0$, therefore for any $y\in g(X)$, we have $\lambda'(y)=\lambda'(1)=\lambda'(y_0)>0$ by Claim 1. Thus by the fact that $g(X)$ is dense in $\mathbb{T}$, we have $\lambda'(\mathbb{T})=\infty$, which contradicts to the fact that $\lambda'$ is a probability measure.\\

With the help of Claim 1 and Claim 2, we have that $\lambda'$ is translation invariant for any open set, thus $\lambda'$ is uniform distribution. Therefore $\lambda'$ coincides with $\lambda/2\pi$ on $\mathbb{T}$ on every open subset and hence on every closed subset of $\mathbb{T}$, which yields that $\lambda'=\lambda/2\pi$ on $\mathbb{T}$. So we have $\lambda(g(X))=2\pi \lambda'(g(X))=2\pi$. 
\end{proof}

\begin{lem}\label{Lemminprim}
Suppose the orders of all elements in $\hat{X}$ are uniformly bounded. Define $r$ to be the minimum integer such that there are infinitely many elements in $\hat{X}$ with order $r$, then $r$ is a prime.
\end{lem}

\begin{proof}[Proof of Lemma \ref{Lemminprim}]
Assume that $r$ is not a prime, then $r=km$ with the integer $k>0$ and a prime $m$, then take a sequence $\{a_i\}_{i=1}^\infty$ such that every $a_i$ has order $r$, then the set defined by $\{a_i^k\}_{i=1}^\infty$ satisfies that every elements in this set has order $m<r$. We claim that the cardinality of this set $\{a_i^k\}_{i=1}^\infty$ is infinity, therefore this contradicts to the fact that $r$ is the minimum, which provides Lemma \ref{Lemminprim}. Suppose that $\{a_i^k\}$'s cardinality is finite then we have a subsequence of $\{a_i\}$ denoted by, without loss of generality, $\{a_i\}$ such that 
$$
a_1^k =a_2^k =a_3^k =\cdots
$$ 
therefore we have $(a_1 a_i^{-1})^k=\hat{e}$ ($\hat{e}$ is the identity in $\hat{X}$) holds for every $i\ge 2$. Further $a_1 a_i\neq 1$ since $a_1 \neq a_i$, $a_1 a_i\neq a_1 a_j$ for $i\neq j$. Therefore we have a set $\{a_1   a_i^{-1}\}_{i=2}^\infty$ with infinitely many elements, all of which have order $k<r$, contradicts to the fact that $r$ is minimum. 
\end{proof}

\begin{lem}\label{IndepEqOrd}
Suppose the orders of all elements in $\hat{X}$ are uniformly bounded. Define $r$ to be the same as in Lemma \ref{Lemminprim}, then $r$ is a prime by Lemma \ref{Lemminprim}. We construct a sequence $\{g_k\}_{k=1}^\infty$ in $\hat{X}$ such that 
\begin{itemize}
\item Every $g_k$ has order $r$.
\item For every $k\ge 2$, $g_k$ is not in the subgroup generated by $g_1,~g_2,\cdots,~g_{k-1}$
\end{itemize}
We can do this because there are infinitely many elements in $\hat{X}$ with order $r$. If we treat every $g_k$ as a random variable: $X\rightarrow\mathbb{T}$, then $g_k$'s are mutually independent and identically distributed.
\end{lem}

\begin{proof}[Proof of Lemma \ref{IndepEqOrd}]
We firstly prove two claims:\\

\textbf{Claim 1.} For every $k$ and $n$ with $n\ge k$, $g_k$ is not in the subgroup generated by $A_{n,k}:=\{g_1,~\cdots,~g_{n}\}\setminus \{g_k\}$. \\

To see this, note that it suffices to prove the claim for $n>k$. Assume that $g_k\in [A_{n,k}]$, then $g_k$ can be written as 
\be\label{ReprOfg_k}
g_k=\prod_{g_i\in A_{n,k}}g_i^{a_i}
\ee
with each $a_i\in [0,r)\cap \mathbb{Z}$. Take the largest index $i$ in (\ref{ReprOfg_k}) such that $a_i\neq 0$, denote this index by $i_0$, thus $i_0>k$ by the definition of $g_k$. So we have that
\be\label{ReprOfg_k2}
g_{i_0}^{a_{i_0}}=g_k\prod_{g_i\in A_{n,k}, ~i<i_0}g_i^{r-a_i}
\ee
Then by the fact that $r$ is a prime, then there exists an integer $b_0$ such that $g_{i_0}^{b_0a_{i_0}}=g_{i_0}$, which means that $g_{i_0}$ is in the subgroup generated by $\{g_1,~\cdots,~g_{i_0-1}\}$ by (\ref{ReprOfg_k2}), contradicts to the definition of $g_{i_0}$, which provides Claim 1.\\

\textbf{Claim 2.} The events $g_k^{-1}(1)$ are mutually independent, i.e.
\be\label{g(1)Indep}
\mathbb{P}\Big(\bigcap_{k=1}^n g_k^{-1}(\{1\})\Big)= \frac{1}{r^n}
\ee

To see this, note that every $g_k$ is a continuous homomorphism from $X$ to $\mathbb{T}$, then $H_n:=\bigcap_{k=1}^n g_k^{-1}(\{1\})$ is an open subgroup of $X$ by Lemma \ref{Rangeofg}. Thus
$$
X=\bigsqcup_{x_i\in X} x_i+H_n
$$
where $x_i +H_n$ are the cosets of $X/H_n$, "$\sqcup$" means disjoint union. Since $\mathbb{P}$ is a Haar measure, we have 
$$
\mathbb{P}(H_n)=\frac{1}{\mbox{number~of~cosets}}
$$ 
Moreover the vector $(g_1,g_2,\cdots,g_n)$ is invariant on each coset $x_i+H_n$ and
$$
(g_1(x_i+H_n),g_2(x_i+H_n),\cdots,g_n(x_i+H_n) )\equiv (g_1(x_i),g_2(x_i),\cdots,g_n(x_i) )
$$
Furthermore if $x_i+H_n\neq x_j+ H_n$, we claim that
\be\label{DistinctVctor}
(g_1(x_i),g_2(x_i),\cdots,g_n(x_i) )\neq (g_1(x_j),g_2(x_j),\cdots,g_n(x_j) )
\ee
To see this, suppose otherwise we have the equality holds in (\ref{DistinctVctor}), then 
$$
(g_1(x_i-x_j),g_2(x_i-x_j),\cdots,g_n(x_i-x_j) )=(1,1,\cdots,1)
$$
therefore $x_i-x_j\in H_n$, contradicts to the fact that $x_iH_n\neq x_j H_n$. Hence the mapping $G:X/H_n\rightarrow [e^{2\pi i/r}]^n\cong \mathbb{F}_r^n$ (since $r$ is prime) defined by 
$$
G(x_i+H_n):=\big(g_1(x_i+H_n),g_2(x_i+H_n),\cdots,g_n(x_i+H_n)\big)
$$
is an injective group homomorphism. It suffices to prove that $G$ is surjective, i.e.
$$
G(X/H_n)=\mathbb{F}_r^n
$$
Suppose for contradiction that $G(X/H_n)$ is a proper subgroup of $\mathbb{F}_r^n$ and therefore is a proper linear subspace of $\mathbb{F}_r^n$. Hence the dimension of $G(X/H_n)$ is $n'<n$. Assume $l$ the number of cosets $\{x_i+H_n\}_{i=1}^{l}$ in $X$. Then consider the matrix
$$
P_n:=\begin{pmatrix}
  g_1(x_1+H_n) & g_2(x_1+H_n) & \cdots & g_n(x_1+H_n) \\
  g_1(x_2+H_n) & g_2(x_2+H_n) & \cdots & g_n(x_2+H_n) \\
  \vdots  & \vdots  & \ddots & \vdots  \\
  g_1(x_l+H_n) & g_2(x_l+H_n) & \cdots & g_n(x_l+H_n)
 \end{pmatrix}
$$
the matrix $P_n$ has rank $n'<n$, then there exists a column with index $k\le n$ that can be written as a linear combination of other column vectors:
$$
\begin{pmatrix}
g_k(x_1+H_n) \\
g_k(x_2+H_n) \\
\vdots  \\
g_k(x_l+H_n)
\end{pmatrix}
=\sum_{i\in [1,n],~i\neq k}a_i
\begin{pmatrix}
g_i(x_1+H_n) \\
g_i(x_2+H_n) \\
\vdots  \\
g_i(x_l+H_n)
\end{pmatrix}
$$
with some $a_1, a_2,\cdots, a_n\in \mathbb{F}_r$. Since $X=\bigsqcup_{i=1}^l x_i+H_n$, thus we have 
$$
g_k=\prod_{g_i\in A_{n,k}}g_i^{a_i}
$$ 
where $A_{n,k}$ is defined in Claim 1. This means that $g_k$ is in the group generated by $A_{n,k}$, contradicts to Claim 1, which ends the proof of Claim 2.\\

Next, note that $g_k's$ share the same image $[e^{2\pi i/r}]$ by Lemma \ref{Rangeofg}. Denote $J$ this image. So it suffices to prove that for every $(a_1,\cdots,a_n)\in J^n$, the probability of the random vector
\begin{eqnarray}\label{MutualyIndp}
\mathbb{P}\big((g_1,g_2,\cdots,g_n)=(a_1,a_2,\cdots,a_n)\big)=\frac{1}{r^n}
\end{eqnarray}
Note that if $(a_1,\cdots,a_n)$ is in the image of the random vector $(g_1,g_2,\cdots,g_n)$, then (\ref{MutualyIndp}) holds by the fact that $\mathbb{P}$ is a Haar measure. On the other hand, if the image  of $(g_1,g_2,\cdots,g_n)$ is a proper subset of $\prod_{k=1}^{n+1}J_k$, then the probability
$$
\mathbb{P}(X)=\mathbb{P}\big((g_1,g_2,\cdots,g_n)\in\prod_{k=1}^{n+1}J_k\big)<1
$$
which yields a contradiction, which ends the proof of Lemma \ref{IndepEqOrd}.
\end{proof}

\begin{lem}\label{LemCardnltOfMn}
Suppose the orders of all elements in $\hat{X}$ are uniformly bounded. Define $r$ and $\{g_k\}_{k=1}^\infty$ the same as in Lemma \ref{IndepEqOrd}, define $M_n$ the subgroup in $\hat{X}$ generated by $g_1,\cdots, g_n$, i.e. 
$$
M_n:=\{g_1^{i_1}g_2^{i_2}\cdots g_n^{i_n}:~0\le i_l\le r-1~~\mbox{for}~~l=1,2,\cdots,n\}
$$
Then the cardinality of $M_n$ is $r^n$. 
\end{lem}

\begin{proof} [Proof of Lemma \ref{LemCardnltOfMn}]
It is equivalent to prove that any two elements $g_1^{i_1}g_2^{i_2}\cdots g_n^{i_n}$ and $g_1^{j_1}g_2^{j_2}\cdots g_n^{j_n}$ are disjoint if the vectors $(i_1,\cdots,i_n)\neq(j_n,\cdots,j_n)$ (i.e. it is a linear space with basis $g_1,\cdots,g_n$). This is equivalent to
$$
g_1^{m_1}g_2^{m_2}\cdots g_n^{m_n}\neq \hat{e}~~\mbox{as~long~as}~~(m_1,m_2,\cdots,m_n)\neq \mbox{zero~vector~in~}\mathbb{F}_r^n
$$
where $\hat{e}$ is the identity in $\hat{X}$. Suppose we have equality holds for some $(m_1,m_2,\cdots,m_n)\neq \mbox{zero~vector}$, without loss of generality, we assume that $m_n\neq 0$. Thus we have
$$
g_n^{r-m_n}=g_1^{m_1}g_2^{m_2}\cdots g_{n-1}^{m_{n-1}}
$$
Therefore $g_n^{r-m_n}$ is in the group generated by $\{g_1,g_2,\cdots,g_{n-1}\}$, since $r$ is prime, then $r-m_n$ and $r$ are coprime, we have two integers $s$ and $t$ such that 
$$
s(r-m_n)+tr=1
$$
Therefore we have $g_n$ is in the group generated by $\{g_1,g_2,\cdots,g_{n-1}\}$, contradicts to the selection of $\{g_k\}$. 
\end{proof}

\begin{lem}\label{H_n^perp}
Based on the conditions of Lemma \ref{Lemminprim}, Lemma \ref{IndepEqOrd} and Lemma \ref{LemCardnltOfMn}, define $H_n^\perp$ the set of all elements in $\hat{X}$ such that their restriction on $H_n$ is trivial. Then 
\begin{eqnarray}
H_n^{\perp}=M_n
\end{eqnarray}
\end{lem}

\begin{proof} [Proof of Lemma \ref{H_n^perp}]
Clearly $H_n^{\perp}\supset M_n$. Assume for contradiction that there exists a $g'\in H_n^{\perp}$ and $g'\notin M_n$. Without loss of generality assume that $g'$ has a prime order $r'$ (i.e. if not then take some power of $g'$). Thus $g'$ is not in the group generated by $g_1,\cdots,g_n$ and $H_n$ is a sugroup of $(g')^{-1}(\{1\})$, then we have that $\mathbb{P}(H_n)$ divides $\mathbb{P}((g')^{-1}(\{1\}))$, thus by Lemma \ref{Distribofg}, we have $1/r^n$ divides $1/r'$, which means that $r'|r^n$, thus $r'=r$ and we have
$$
H_n=\Big(\bigcap_{k=1}^n g_k^{-1}(\{1\})\Big)\cap (g')^{-1}(\{1\})
$$
Then apply Lemma \ref{IndepEqOrd} to the last equality of the following equation:
$$
\frac{1}{r^n}=\mathbb{P}(H_n)=\mathbb{P}\Bigg(\Big(\bigcap_{k=1}^n g_k^{-1}(\{1\})\Big)\cap (g')^{-1}(\{1\}) \Bigg)=\frac{1}{r^{n+1}}
$$
which yields a contradiction. 
\end{proof}

We have a direct conclusion from Lemma \ref{H_n^perp}.

\begin{lem}\label{LemNumbofCosetsofHn^perp}
The cardinality of $\hat{X}/H_n^{\perp}$ is infinity. 
\end{lem}

\begin{lem}\label{DualRestrictionInt}
Suppose that $X$ is a compact LCA group with the probability Haar measure, for any closed subgroup $H<X$ with positive probability and $\hat{x}\in \hat{X}$ such that the restriction $\hat{x}\big|_H$ is not trivial, then we have 
$$
\int_H\hat{x}(x)\mathbb{P}(dx)=0
$$ 
\end{lem}

\begin{proof} [Proof of Lemma \ref{DualRestrictionInt}]
Note that the restriction of $\hat{x}$ on $H$ is an element of $\hat{H}$ and is not identity, and it is also easy to verify that the restriction of $\mathbb{P}$ on $H$ is still a Haar measure. By the fact that $H$ is a compact LCA group, then the orthogonality of the elements in $\hat{H}$ (see Proposition \ref{Ortho}) provides the lemma.
\end{proof}

\begin{proof} [Proof of Proposition \ref{PropR2XCom}, Case 1]
Assume that $g$ is the element in $\hat{X}$ with order infinity. Set $G$ to be the subgroup in $\hat{X}$ generated by $g$. Define for $k\ge 0$,
$$
U_k:=\{x\in X:~g(x)\in(e^{-\pi i/(3k)},e^{\pi i/(3k)})\}
$$
Thus $U_k$ is decreasing (i.e. $U_{k+1}\subset U_k$) and for every $1\le j\le k$, $g^j(U_k)=(g(U_k))^j\subset(e^{-\pi i/3},e^{-\pi i/3})$. Moreover $U_k=-U_k$ by the fact that $g(-x)=(g(x))^{-1}$. Hence we have
\begin{eqnarray}\label{Re(g_j)>1/2}
\mbox{For~every~element~} x\in U_k,~Re(g^j(-x))>1/2~\mbox{for}~ 1\le j\le k
\end{eqnarray}
Moreover, by Lemma \ref{Rangeofg}, we have $\mathbb{P}(U_k)=\frac{1}{3k}$. Define $f_k(x):=\frac{1}{\mathbb{P}(U_k)}\mathbb{I}_{U_k}(x)$, where $\mathbb{I}_{U_k}(x)$ is the indicator function of $U_k$, then $\|f_k\|_p=(3k)^{1-1/p}$. On the other hand, by Proposition \ref{HarMeas},
\begin{eqnarray*}
\|\widehat{f_k}\|_q &=&\Big(\int_{\hat{X}}|\widehat{f_k}(\hat{x})|^q\hat{\alpha}(d\hat{x})\Big)^{1/q}
\ge \Big(\sum_{j=1}^k |\widehat{f_k}(g^j)|^q \Big)^{1/q}\\
&\ge &\Big( \sum_{j=1}^k \Big|\int_{U_k}\frac{1}{\mathbb{P}(U_k)}g^j(-x)\alpha(dx)\Big|^q \Big)^{1/q}\\
&\ge &\Big( \sum_{j=1}^k \Big|\int_{U_k}\frac{1}{\mathbb{P}(U_k)}Re\big(g^j(-x)\big)\alpha(dx)\Big|^q \Big)^{1/q}\\
&\ge & \frac{k^{1/q}}{2}~~(\mbox{by~(\ref{Re(g_j)>1/2})})
\end{eqnarray*}
which yields
$$
\frac{\|\widehat{f_k}\|_q }{\|f_k\|_p}\ge \frac{3^{1/p-1}}{2}k^{1/p+1/q-1}\rightarrow\infty
$$
as $k\rightarrow\infty$, which ends the proof of Case 1. 
\end{proof}

\begin{proof} [Proof of Proposition \ref{PropR2XCom}, Case 2]
Take a sequence $\{g_n\}$ in $\hat{X}$ such that each $g_n$ has order $m_n$ and $m_n\nearrow\infty$. Define 
$$
U_n:=\{x\in X:~g_n(x)=1\}
$$
Thus $U_n=-U_n$ by the fact that $U_n$ is a subgroup of $X$. Hence we have
\begin{eqnarray}\label{g_n^j=1onUn}
\mbox{For~every~}x\in U_n,~g_n^j(-x)=1~\mbox{for~} 1\le j\le m_n
\end{eqnarray} 
Then by Lemma \ref{Distribofg}, $\mathbb{P}(U_n)= 1/m_n$. Define $f_n:=\frac{1}{\mathbb{P}(U_n)}\mathbb{I}_{(U_n)}$. We have $\|f_n\|_p=(m_n)^{1-1/p}$. On the other hand
\begin{eqnarray*}
\|\widehat{f_n}\|_q&=&\Big(\int_{\hat{X}}|\widehat{f_n}(\hat{x})|^q\hat{\alpha}(d\hat{x})\Big)^{1/q}
\ge \Big( \sum_{j=1}^{m_n} |\widehat{f_n}(g_n^j)|^q \Big)^{1/q}\\
&\ge & \Big( \sum_{j=1}^{m_n} \Big|\int_{U_n}\frac{1}{\mathbb{P}(U_n)}g_n^{j}(-x)\alpha(dx)\Big|^q \Big)^{1/q}\\
&=& \Big( \sum_{j=1}^{m_n} \Big|\int_{U_n}\frac{1}{\mathbb{P}(U_n)}\alpha(dx)\Big|^q \Big)^{1/q}~~(\mbox{by (\ref{g_n^j=1onUn})})\\
&=&m_n^{1/q}
\end{eqnarray*}
which yields
$$
\frac{\|\widehat{f_n}\|_q}{\|f_n\|_p}\ge (m_n)^{1/p+1/q-1}\rightarrow\infty
$$
as $n\rightarrow\infty$, which ends the proof of Case 2. 
\end{proof}

\begin{proof} [Proof of Proposition \ref{PropR2XCom}, Case 3]
Suppose the orders of all elements in $\hat{X}$ are uniformly bounded. Define $r$, $\{g_k\}_{k=1}^\infty$, $H_n$ and $M_n$ the same as in Lemma \ref{Lemminprim}, Lemma \ref{IndepEqOrd}, Lemma \ref{LemCardnltOfMn} and Lemma \ref{H_n^perp}. Define a sequence of functions $f_n:=\frac{1}{\mathbb{P}(H_n)}\mathbb{I}_{H_n}$, then $\|f_n\|_p=\mathbb{P}(H_n)^{1/p-1}=(r^n)^{1-1/p}$ by Lemma \ref{IndepEqOrd}. On the other hand, we have, by Proposition \ref{HarMeas}
\begin{eqnarray}\label{qNormUnifBdd}
\|\widehat{f_n}\|_q=	\Big(\int_{\hat{X}}|\widehat{f_n}(\hat{x})|^q\hat{\mathbb{P}}(d\hat{x})\Big)^{1/q}\ge \Big(\sum_{g\in M_n}|\widehat{f_n}(g)|^q\Big)^{1/q}
\end{eqnarray}
Then by Lemma \ref{H_n^perp}, for every $g\in M_n$, $g$ is trivial on $H_n$. So we have that
$$
\widehat{f_n}(g)=\frac{1}{\mathbb{P}(H_n)}\int_{H_n}g(-x)\mathbb{P}(dx)=1
$$
Therefore by (\ref{qNormUnifBdd}) and Lemma \ref{LemCardnltOfMn}, we have
$$
\|\widehat{f_n}\|_q\ge \Big(\sum_{g\in M_n}|\widehat{f_n}(g)|^q\Big)^{1/q}=(r^n)^{1/q}
$$
and
$$
\frac{\|\widehat{f_n}\|_q}{\|f_n\|_p}\ge (r^n)^{1/p+1/q-1}\rightarrow\infty
$$
as $n\rightarrow\infty$, what ends the proof of Case 3 and provides Proposition 4.1.2.
\end{proof}

\begin{prop}\label{PropR3XCom}
If $(1/p,1/q)\in R_3$, then $C_{p,q}= \infty$.
\end{prop}

We will actually prove a stronger conclusion: $C_{p,q}= \infty$ for all $0<q<2$ and $p\ge 1$. We still consider the three cases provided in the proof of Proposition \ref{PropR2XCom}. 

\begin{proof}[Proof of Proposition \ref{PropR3XCom}, Case 1]
Assume that $g$ is the element in $\hat{X}$ with order infinity. Set $G:=\{g^n,n\in \mathbb{Z}\}$ to be the subgroup in $\hat{X}$ generated by $g$. Define for every $x\in X$,
$$
g(x):=\exp\{ib(x)\}
$$
where $b(x):X\rightarrow[0,2\pi]$ is a continuous function. Therefore $g^n(x)=\exp\{inb(x)\}$. Then by \cite{Zyg59:book}, Volume 1, p. 199, (4-9) Theorem, we define, for $\beta>1$ and $c$ positive, the function
\be\label{functionLp<infty}
h(y):=\sum_{n=2}^\infty\frac{e^{icn\log n}}{n^{1/2}(\log n)^\beta}e^{iny}
\ee
uniformly converges for $y\in [0,2\pi]$, whereas its Fourier coefficients are not in $l_q$ for any $q<2$. Thus $f(x)$ defined by 
$$
f(x):=h(b(x))=\sum_{n=2}^\infty\frac{e^{icn\log n}}{n^{1/2}(\log n)^\beta}e^{inb(x)}=\sum_{n=2}^\infty\frac{e^{icn\log n}}{n^{1/2}(\log n)^\beta}g^n(x)
$$
is well defined and is continuous and bounded on $X$. Therefore $f\in L_p$ for all $p\ge 1$. On the other hand $\{\hat{f}(g^n)\}$ is not in $l_q$ for any $q<2$, which ends the proof of Case 1. 
\end{proof}

\begin{proof}[Proof of Proposition \ref{PropR3XCom}, Case 2]
Suppose that the orders of the elements of $\hat{X}$ are not bounded. So there is a sequence of elements $g_n$ in $\hat{X}$ such that the order of $g_n$ is $m_n\nearrow\infty$. Similarly we define $g_n(x):=\exp\{ib_n(x)\}$ for some continuous function $b_n:X\rightarrow[0,2\pi]$. And define 
$$
f_{n}(x):=\sum_{k=2}^{m_n-1}\frac{e^{ick\log k}}{k^{1/2}(\log k)^\beta}e^{ikb_n(x)}=\sum_{k=2}^{m_n-1}\frac{e^{ick\log k}}{k^{1/2}(\log k)^\beta}g_n^k(x)
$$ 
Thus $f_n(x)$ is uniformly bounded for any $n$ and $x\in X$ by the fact that $h(x)$ defined in (\ref{functionLp<infty}) converges uniformly. Thus $\|f_n\|_p$ is uniformly bounded for any $n$, and $\|\widehat{f_n}\|_q\rightarrow\infty$ as $n \rightarrow\infty$ for any $q<2$. In other words, $C_{p,q}=\infty$ for this case. 
\end{proof}

\begin{proof}[Proof of Proposition \ref{PropR3XCom}, Case 3]
Suppose that the orders of the elements of $\hat{X}$ are uniformly bounded. Then we define $r$, $\{g_k\}_{k=1}^\infty$, $\{H_n\}$ and $\{M_n\}$ the same as in the proof of Proposition \ref{PropR2XCom}, Case 3. Apply Lemma \ref{LemNumbofCosetsofHn^perp}, we can find $r^n$ cosets in $\hat{X}/H_n^{\perp}$, denote $\{\gamma_k\}_{k=1}^{r^n}$ a set of the representants of these $r^n$ cosets with $k=1,2,\cdots,r^n$. Define 
$$
f_n(x):=\sum_{k=1}^{r^n}\gamma_k(x)\mathbb{I}_{x_k+H_n}(x)
$$
where $x_k+H_n$ are the cosets of $X/H_n$. Then it is easy to see that $|f_n|\equiv 1$, therefore $\|f_n\|_p=1$. On the other hand, for any $\gamma\in \gamma_{k_0}H_n^{\perp}$ with some $k_0$, 
\begin{eqnarray}\label{FTofBiunimodularfn}
\widehat{f_n}(\gamma)=\sum_{k=1}^{r^n}\int_X\gamma_k(x)\mathbb{I}_{x_k+H_n}(x)\gamma(-x)\mathbb{P}(dx)
\end{eqnarray}
For $k_0$, we have
\begin{eqnarray*}
&&\int_X\gamma_{k_0}(x)\mathbb{I}_{x_{k_0}+H_n}(x)\gamma(-x)\mathbb{P}(dx)\\
&=&\int_{x_{k_0}+H_n}\big(\gamma_{k_0}\gamma^{-1}\big)(x)\mathbb{P}(dx)\\
&=&\big(\gamma_{k_0}\gamma^{-1}\big)(x_{k_0})\int_{H_n}\big(\gamma_{k_0}\gamma^{-1}\big)(x)\mathbb{P}(dx)\\
&=&\big(\gamma_{k_0}\gamma^{-1}\big)(x_{k_0})\frac{1}{r^n}
\end{eqnarray*}
where the last equality is by the fact that $\gamma\in \gamma_{k_0}H_n^{\perp}\Rightarrow\gamma_{k_0}\gamma^{-1}\in H_n^{\perp}$. Further if $k\neq k_0$, we have 
\begin{eqnarray*}
&&\int_X\gamma_{k}(x)\mathbb{I}_{x_{k}+H_n}(x)\gamma(-x)\mathbb{P}(dx)\\
&=&\int_{x_{k}+H_n}\big(\gamma_{k}\gamma^{-1}\big)(x)\mathbb{P}(dx)\\
&=&\big(\gamma_{k}\gamma^{-1}\big)(x_{k})\int_{H_n}\big(\gamma_{k}\gamma^{-1}\big)(x)\mathbb{P}(dx)=0
\end{eqnarray*}
where the last equality is by the fact that $\gamma_k\gamma^{-1}$ is non-trivial on $H_n$ and therefore applying Lemma \ref{DualRestrictionInt}. Thus (\ref{FTofBiunimodularfn}) equals
\begin{eqnarray}\label{FTofBiunimodularfn2}
\widehat{f_n}(\gamma)=\begin{cases} (\gamma_{k}\gamma^{-1})(x_k)\frac{1}{r^n} & \mbox{if}~\gamma\in \gamma_kH_n^{\perp}~\mbox{for~some~}k\\
0 & \mbox{otherwise}
\end{cases}
\end{eqnarray}
Denote $Q_n:=\bigsqcup_{k=1}^{r^n} \gamma_k H_n^{\perp}$, therefore $|Q_n|=r^{2n}$ by Lemma \ref{LemCardnltOfMn} and Lemma \ref{H_n^perp}. Then by (\ref{FTofBiunimodularfn2})
\begin{eqnarray}\label{|FTofBiunimodularfn|}
|\widehat{f_n}(\gamma)|=\begin{cases} \frac{1}{r^n} & \mbox{if}~\gamma\in Q_n\\
0 & \mbox{otherwise}
\end{cases}
\end{eqnarray}
which yields
\begin{eqnarray*}
\|\widehat{f_n}\|_q &\ge & \Big(\sum_{\gamma\in Q_n}\frac{1}{r^{nq}}\Big)^{1/q} =r^{\frac{(2-q)n}{q}}\rightarrow\infty
\end{eqnarray*}
as $n\rightarrow\infty$, which ends the proof of Case 3 and provides Proposition \ref{PropR3XCom}. Proposition \ref{PropR1XCom}, \ref{PropR2XCom} and \ref{PropR3XCom} together provide Theorem 3.1.
\end{proof}

\begin{rmk}\label{RmkIndpPrim}
We have a similar conclusion as Lemma \ref{IndepEqOrd}, even though we will not use this conclusion, it is still an interesting conclusion and its proof is stronger than the proof of Lemma \ref{IndepEqOrd}. Suppose $X$ is a compact LCA group associated with a probability Haar measure $\mathbb{P}$ and $\hat{X}$ is the discrete dual group, and suppose that we can take a sequence $\{g_k\}_{k=1}^\infty$ such that their orders are distinct primes, then $g_k$'s are mutually independent.
\end{rmk}

\begin{proof} [Proof of Remark \ref{RmkIndpPrim}]
We assume that each $g_k$ has prime order $r_k$, and $r_k$'s are distinct. We firstly prove that the events $g_k^{-1}(1)$ are mutually independent. It suffices to prove that 
\begin{eqnarray}\label{IndepDifOrder}
\mathbb{P}\Big(\bigcap_{k=1}^n g_k^{-1}(\{1\})\Big)= \frac{1}{\prod_{k=1}^nr_k}
\end{eqnarray}
We will use induction, suppose that (\ref{IndepDifOrder}) holds for $n$, then for $n+1$, define $H_n$ the subgroup in $X$ such that
\begin{eqnarray}
H_n:=\bigcap_{k=1}^n g_k^{-1}(\{1\})
\end{eqnarray}
Therefore $H_n$ is a compact subgroup of $X$. We treat define $\tilde{g}_{n+1}:=g_{n+1}\big|_{H_n}$, i.e. the restriction of $g_{n+1}$ on $H_n$. Therefore we have that $\tilde{g}_{n+1}^{r_{n+1}}$ is trivial, therefore the order of $\tilde{g}_{n+1}$ divides $r_{n+1}$, hence the order of $\tilde{g}_{n+1}$ is either 1 or $r_{n+1}$ by the fact that $r_{n+1}$ is prime. If the order of $\tilde{g}_{n+1}$ is 1, then $H_n\subset g_{n+1}^{-1}(\{1\})$. So we have that 
\begin{eqnarray}\label{H_nSubset}
H_n\subset \left(g_1g_2\cdots g_ng_{n+1}^{-1}\right)^{-1}(\{1\})
\end{eqnarray}
where the $g_{n+1}^{-1}$ is the inverse of $g_{n+1}$. On the other hand, we have the order of $g_1g_2\cdots g_ng_{n+1}^{-1}$ is $\prod_{k=1}^{n+1}r_k$, therefore by Lemma \ref{Distribofg}, 
\begin{eqnarray}\label{ProbabilityOfrhs}
\mathbb{P} \left\{\big(g_1g_2\cdots g_ng_{n+1}^{-1}\big)^{-1}(\{1\})\right\}=\frac{1}{\prod_{k=1}^{n+1}r_k}
\end{eqnarray}
However by induction hypothesis, 
\begin{eqnarray}\label{ProbabilityOflhs}
\mathbb{P}(H_n)=\frac{1}{\prod_{k=1}^{n}r_k}
\end{eqnarray}
So by (\ref{H_nSubset}), (\ref{ProbabilityOfrhs}) and (\ref{ProbabilityOflhs}) we have 
$$
\frac{1}{\prod_{k=1}^{n}r_k}\le \frac{1}{\prod_{k=1}^{n+1}r_k}
$$
which leads to a contradiction. So we have that the order of $\tilde{g}_{n+1}$ is $r_{n+1}$. Next, we need to verify that the probability measure on $H_n$ definde by the conditional probability $\mathbb{P}(\cdot|H_n)$ is still a Haar measure in order to use Lemma \ref{Distribofg}. To see this, it suffices to verify that $\mathbb{P}(\cdot|H_n)$ is translation invariant. For any measurable set $M\subset H_n$ and any element $\hat{a}\in H_n$, we have
\begin{eqnarray*}
\mathbb{P}(\hat{a}M|H_n)=\frac{\mathbb{P}(\hat{a}M\cap H_n)}{\mathbb{P}(H_n)}=\frac{\mathbb{P}(\hat{a}(M\cap H_n))}{\mathbb{P}(H_n)}=\frac{\mathbb{P}(M\cap H_n)}{\mathbb{P}(H_n)}=\mathbb{P}(M|H_n)
\end{eqnarray*}
where the second equality is by the fact that $\hat{a}\in H_n\Rightarrow \hat{a}H_n=H_n$. The third equality is by the fact that $\mathbb{P}$ is a Haar measure. Therefore by applying Lemma \ref{Distribofg} again on $H_n$, $\mathbb{P}(\cdot|H_n)$ and $\tilde{g}_{n+1}$, we have that $\mathbb{P}\left(\tilde{g}_{n+1}^{-1}(\{1\})|H_n\right)$ is $1/r_{n+1}$. So we have that
\begin{eqnarray}
\mathbb{P}(H_{n+1})=\mathbb{P}(H_n)\mathbb{P}\big(\tilde{g}_{n+1}^{-1}(\{1\})\big|H_n\big)=\frac{1}{\prod_{k=1}^{n+1}r_k}
\end{eqnarray}
which provides conclusion (\ref{IndepDifOrder}). Next, to prove that $g_k$'s are mutually independent, we denote $J_k$ the range of $g_k$ on $\mathbb{T}$. So it suffices to prove that for every $(a_1,\cdots,a_n)\in \prod_{k=1}^{n+1}J_k$, the probability of the random vector
\begin{eqnarray}\label{IndpDiffOrderRange}
\mathbb{P}\big((g_1,g_2,\cdots,g_n)=(a_1,a_2,\cdots,a_n)\big)=\frac{1}{\prod_{k=1}^{n+1}r_k}
\end{eqnarray}
Note that if $(a_1,\cdots,a_n)$ is in the image of the random vector $(g_1,g_2,\cdots,g_n)$, then (\ref{IndpDiffOrderRange}) holds by the fact that $\mathbb{P}$ is a Haar measure. On the other hand, if the image  of $(g_1,g_2,\cdots,g_n)$ is a proper subset of $\prod_{k=1}^{n+1}J_k$, then the probability
$$
\mathbb{P}(X)=\mathbb{P}\big((g_1,g_2,\cdots,g_n)\in\prod_{k=1}^{n+1}J_k\big)<1
$$
which yields a contradiction and ends the proof.
\end{proof}

\section{The discrete case}
\label{sec:discrete}

\begin{thm}\label{ThmXDis}
If $X$ is a discrete non-finite LCA group, we consider the three regions as in Figure 2:
\begin{eqnarray}
&&R_1':=\Big\{\Big(\frac{1}{p},\frac{1}{q}\Big)\in [0,\infty)^2:\frac{1}{p}+\frac{1}{q}< 1, \frac{1}{q}< \frac{1}{2}\Big\}\\
&&R_2':=\Big\{\Big(\frac{1}{p},\frac{1}{q}\Big)\in [0,\infty)^2:\frac{1}{p}+\frac{1}{q}\ge 1, \frac{1}{p}\ge \frac{1}{2}\Big\}\\
&&R_3':=\Big\{\Big(\frac{1}{p},\frac{1}{q}\Big)\in [0,1]^2:\frac{1}{p}< \frac{1}{2}, \frac{1}{q}\ge \frac{1}{2}\Big\}
\end{eqnarray} 
Then the norm of the Fourier operator from $L^p(X)$ to $L^q(\hat{X})$ satisfies
\begin{eqnarray*}
C_{p,q}=\begin{cases}
\infty & \mbox{if}~\Big(\frac{1}{p},\frac{1}{q}\Big)\in R_1'\cup R_3'\\
\hat{\alpha}(\hat{X})^{1/p+1/q-1} & \mbox{if}~\Big(\frac{1}{p},\frac{1}{q}\Big)\in R_2'
\end{cases}
\end{eqnarray*}
where $\hat{\alpha}$ is the Haar measure of the dual group of $X$ normalized by Theorem \ref{InvFT}. 
\end{thm}


\begin{prop}\label{PropR2'XDis}
If $(1/p,1/q)\in R_2'$, then $C_{p,q}=\hat{\alpha}(\hat{X})^{1/p+1/q-1}$.
\end{prop}

\begin{proof} [Proof of Proposition \ref{PropR2'XDis}]
We divide $R_2'$ into two parts: $R_{21}':= R_2'\cap \{1/q\le 1\}$ and $R_{22}':=R_2'\setminus R_{21}'$ as in Figure 3.
\begin{figure}[H]
	\centering
	\begin{tikzpicture}[scale=1.5,dot/.style={draw,circle,inner sep=2pt,fill},lab/.style={outer sep=2pt,font=\large}]
		\draw [thick,->](0,0)--(0,5);
		\draw [thick,->](0,0)--(5,0);
		\draw (2,4)--(5,4);
		\draw (2,2)--(4,0);
		\draw (2,2)--(2,5);
		\draw (2,2)--(0,2);
		\node[dot] at (0,0) {} node[lab] at (0,0) [below left] {$0$};
		\node[dot] at (2,0) {} node[lab] at (2,0) [below] {$\frac{1}{2}$};
		\node[dot] at (4,0) {} node[lab] at (4,0) [below] {$1$} node[lab] at (5,0) [above right] {$\frac{1}{p}$};
		\node[dot] at (0,2) {} node[lab] at (0,2) [left] {$\frac{1}{2}$};
		\node[dot] at (0,4) {} node[lab] at (0,4) [left] {$1$} node[lab] at (0,5) [above right] {$\frac{1}{q}$};
		\node[dot] at (2,2) {};
		\node[dot] at (2,4) {};
		\node[lab] at (1,3.5) {$R_3'$};
		
		\node[lab] at (4,2.1) {$R_{21}'$};
		
		\node[lab] at (1.6,1) {$R_1'$};
		
		\node[lab] at (3.5,4.5) {$R_{22}'$};
		
	\end{tikzpicture}
	\caption{$X$ discrete LCA group}
	\label{figure:2}
\end{figure}
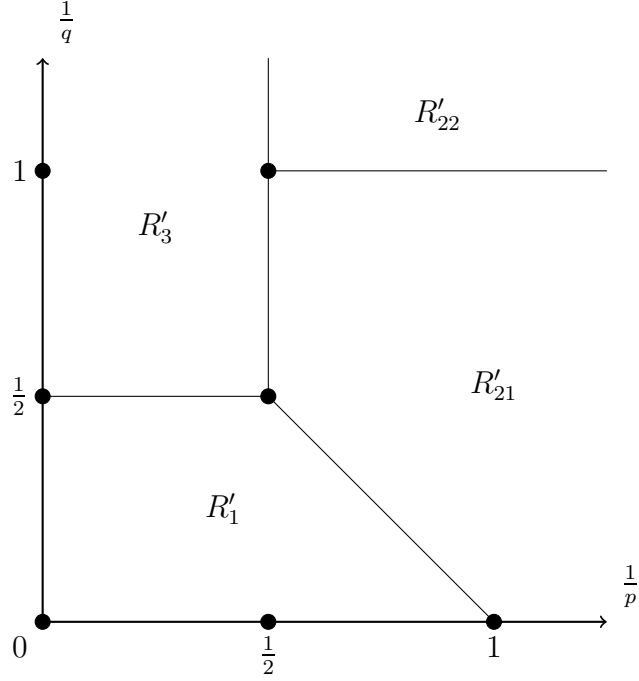
We will firstly prove that $C_{p,q}\le \hat{\alpha}(\hat{X})^{1/p+1/q-1}$ for all $(1/p,1/q)\in R_2'$ and then prove that this upper bound can be attained. Firstly consider the region $R_{21}'$, which is a subset of $(1/p,1/q)\in [0,\infty)\times[0,1]$ and hence we can use Riesz-Thorin Theorem and Corollary \ref{corRT}. We firstly prove that $C_{p,q}\le \hat{\alpha}(\hat{X})^{1/p+1/q-1}$ for all $(1/p,1/q)\in R_{21}'$. Note that it suffices to prove that for any $p_0$ with $p_0<1$, we have $C_{p,q}\le \hat{\alpha}(\hat{X})^{1/p+1/q-1}$ for all $(1/p,1/q)\in R_{21,p_0}':=R_{21}'\cap \{1/p\le 1/{p_0}\}$ as in Figure 4.

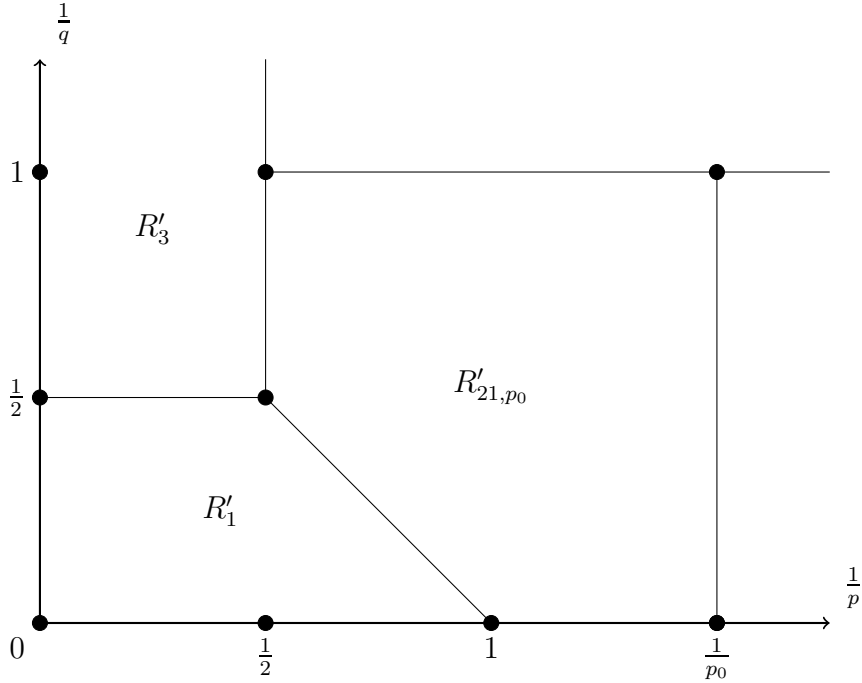
\begin{figure}[H]
	\centering
	\begin{tikzpicture}[scale=1.5,dot/.style={draw,circle,inner sep=2pt,fill},lab/.style={outer sep=2pt,font=\large}]
		\draw [thick,->](0,0)--(0,5);
		\draw [thick,->](0,0)--(7,0);
		\draw (2,4)--(7,4);
		\draw (2,2)--(4,0);
		\draw (2,2)--(2,5);
		\draw (2,2)--(0,2);
		\draw (6,0)--(6,4);
		\node[dot] at (0,0) {} node[lab] at (0,0) [below left] {$0$};
		\node[dot] at (2,0) {} node[lab] at (2,0) [below] {$\frac{1}{2}$};
		\node[dot] at (4,0) {} node[lab] at (4,0) [below] {$1$} node[lab] at (7,0) [above right] {$\frac{1}{p}$};
		\node[dot] at (6,0) {} node[lab] at (6,0) [below] {$\frac{1}{p_0}$};
		\node[dot] at (0,2) {} node[lab] at (0,2) [left] {$\frac{1}{2}$};
		\node[dot] at (0,4) {} node[lab] at (0,4) [left] {$1$} node[lab] at (0,5) [above right] {$\frac{1}{q}$};
		\node[dot] at (2,2) {};
		\node[dot] at (2,4) {};
		\node[dot] at (6,0) {};
		\node[dot] at (6,4) {};
	
		\node[lab] at (1,3.5) {$R_3'$};
		
		\node[lab] at (4,2.1) {$R_{21,p_0}'$};
		
		\node[lab] at (1.6,1) {$R_1'$};
		
		
	\end{tikzpicture}
	\caption{$X$ discrete LCA group}
	\label{figure:2}
\end{figure}

Since the region $R_{21,p_0}'$ is convex and 
$$
R_2'=\mbox{hull}\Big((1,0),(1/{p_0},0),(1/{p_0},1),(1/2,1),(1/2,1/2)\Big)
$$
and $\log \hat{\alpha}(\hat{X})^{1/p+1/q-1}$ is affine for $(1/p,1/q)$, it suffices to check that $C_{p,q}\le \hat{\alpha}(\hat{X})^{1/p+1/q-1}$ holds at $(1,0),(1/{p_0},0),(1/{p_0},1),(1/2,1),(1/2,1/2)$ by Corollary \ref{corRT}.\\

At the point $(1,0)$, we have
\begin{eqnarray*}
\|\hat{f}\|_\infty &\le & \sup_{\gamma\in\hat{X}}|\hat{f}(\gamma)|=\sup_{\gamma\in\hat{X}}|\int_Xf(x)\gamma(-x)\alpha(dx)|
\le \int_X|f(x)|\alpha(dx)\\
&=&\|f\|_1
\end{eqnarray*}
so $C_{1,\infty}\le 1$. \\

At the point $(1/{p_0},0)$, without loss of generality, we assume that $\|f\|_{p_0}=\hat{\alpha}(\hat{X})^{-1/{p_0}}$, therefore $|f(x)|\le 1$ for all $x\in X$. We have, 
\begin{eqnarray*}
\|\hat{f}\|_{\infty}\le \|f\|_1=\frac{1}{\hat{\alpha}(\hat{X})}\sum_{x\in X}|f(x)|\le \frac{1}{\hat{\alpha}(\hat{X})}\sum_{x\in X}|f(x)|^{p_0}=\|f\|_{p_0}^{p_0}=\hat{\alpha}(\hat{X})^{1/{p_0}-1}\|f\|_{p_0}
\end{eqnarray*}
where the inequality is by the fact that $|f(x)|\le 1$ for all $x\in X$ and the fact that $p_0<1$. So $C_{p_0,\infty}\le \hat{\alpha}(\hat{X})^{1/{p_0}-1}$.\\


At the point $(1/{p_0},1)$, we have
\begin{eqnarray*}
\|\hat{f}\|_1\le \hat{\alpha}(\hat{X})\|\hat{f}\|_\infty\le \hat{\alpha}(\hat{X})C_{p_0,\infty}\|f\|_{p_0}=\hat{\alpha}(\hat{X})^{1/{p_0}}\|f\|_{p_0}
\end{eqnarray*} 
therefore $C_{p_0,1}\le \hat{\alpha}(\hat{X})^{1/{p_0}}$.\\

At the point $(1/2,1)$, we have, by Cauchy-Schwartz Inequality and Proposition \ref{Unitary}
\begin{eqnarray*}
\|\hat{f}\|_1&\le& \hat{\alpha}(\hat{X})^{1/2}\|\hat{f}\|_2=\hat{\alpha}(\hat{X})^{1/2}\|f\|_2
\end{eqnarray*}
So $C_{2,1}\le \hat{\alpha}(\hat{X})^{1/2}$.\\

At the point $(1/2,1/2)$, we have that $C_{2,2}=1$ by Proposition \ref{Unitary}, which provides $C_{p,q}\le \hat{\alpha}(\hat{X})^{1/p+1/q-1}$ for all $(1/p,1/q)\in R_{21}'$.\\

Next, consider the region $R_{22}'$, we have that for any $(1/p,1/q)\in R_{22}'$, by Cauchy-Schwartz Inequality,
\begin{eqnarray}\label{eq:useCp1norm}
\|\hat{f}\|_q\le \hat{\alpha}(\hat{X})^{1/q-1}\|\hat{f}\|_1\le \hat{\alpha}(\hat{X})^{1/q-1}C_{p,1}\|f\|_p
\end{eqnarray}
We have that $(1/p,1)\in R_{21}'$, hence by our conclusion in $R_{21}'$, $C_{p,1}\le \hat{\alpha}(\hat{X})^{1/p}$, therefore by (\ref{eq:useCp1norm}), we have
$$
\|\hat{f}\|_q\le \hat{\alpha}(\hat{X})^{1/p+1/q-1}\|f\|_p
$$
which provides that $C_{p,q}\le \hat{\alpha}(\hat{X})^{1/p+1/q-1}$ for all $(1/p,1/q)\in R_2'$.\\

Further, the upper bound $\hat{\alpha}(\hat{X})^{1/p+1/q-1}$ can be attained by any time basis (i.e. delta functions on $X$). Assume $f(x):=\mathbb{I}_{x_0}(x)$ for some $x_0\in X$, then $\|f\|_p=\hat{\alpha}(\hat{X})^{-1/p}$, and for any $\gamma\in\hat{X}$, 
$$
\hat{f}(\gamma)=\frac{1}{\hat{\alpha}(\hat{X})}\gamma(-x_0)
$$
which yields $\|\hat{f}\|_q=\hat{\alpha}(\hat{X})^{1/q-1}$, which provides $C_{p,q}\ge \hat{\alpha}(\hat{X})^{1/p+1/q-1}$ for all $(p,q)\in R_2'$ and ends the proof of Proposition \ref{PropR2'XDis}.
\end{proof}

\begin{prop}\label{PropR1'XDis}
If $(1/p,1/q)\in R_1'$, then $C_{p,q}=\infty$. 
\end{prop}

Since $X$ is discrete, as in the proof of Proposition \ref{PropR2XCom}, we have three cases: 
\begin{itemize}
\item Case 1. ${X}$ has an element with order infinity.
\item Case 2. Every element of ${X}$ has a finite order, and the order set is not bounded.
\item Case 3. Every element of ${X}$ has a finite order, and the order set is uniformly bounded.
\end{itemize}

Similarly, we assume that the Haar measure on $\hat{X}$ is a probability measure $\mathbb{P}$.\\

\begin{proof} [Proof of Proposition \ref{PropR1'XDis}, Case 1]
Assume that $g$ is the element in $X$ with order infinity, set $G:=[g]$, define:
$$
U_k:=\{\hat{x}\in\hat{X}:~\hat{x}(g)\in (e^{-\pi i/(3k)},e^{\pi i/(3k)})\}
$$ 
Then similarly $U_k$ is decreasing. We treat $g$ as an element in $\hat{\hat{X}}$, note that $g(\hat{x}):=\hat{x}(g)$ where $g\in X\simeq \hat{\hat{X}}$. Hence Lemma \ref{Rangeofg} and Lemma \ref{Distribofg} are valid by substituting $X$ by $\hat{X}$. Then we define a sequence of functions $f_k$ on $X$ by 
$$
f_k:=\sum_{j=1}^k \mathbb{I}_{g^j}
$$
Thus $\|f_k\|_p=k^{1/p}$. On the other hand, for every $\hat{x}\in\hat{X}$, 
$$
\widehat{f_k}(\hat{x})=\sum_{j=1}^k \hat{x}(g^{-j})=\sum_{j=1}^k (g^j(\hat{x}))^{-1}
$$
Moreover, we have for every $\hat{x}\in U_k$, and $1\le j\le k$, 
\begin{eqnarray}\label{Reg^j>1/2}
Re\left(\left(g^j(\hat{x})\right)^{-1}\right)=Re\left(\hat{x}(g^{-j})\right)=Re\left(\left(\hat{x}(g)\right)^{-j}\right)>1/2
\end{eqnarray}
which yields
\begin{eqnarray*}
\|\widehat{f_k} \|_q &=& \left(\int_{\hat{X}} \Big| \sum_{j=1}^k \left(g^j(\hat{x})\right)^{-1}\Big|^q \hat{\alpha}(d\hat{x})\right)^{1/q}
\ge  \left(\int_{U_k} \Big| \sum_{j=1}^k \left(g^j(\hat{x})\right)^{-1}\Big|^q \hat{\alpha}(d\hat{x})\right)^{1/q}\\
&\ge & \left(\int_{U_k} \Big| \sum_{j=1}^k Re\left((g^j(\hat{x}))^{-1}\right)\Big|^q \hat{\alpha}(d\hat{x})\right)^{1/q}\\
&\ge & \frac{k\mathbb{P}(U_k)^{1/q}}{2}~~(\mbox{by~(\ref{Reg^j>1/2})})\\
&= & \frac{k^{1-1/q}}{2\cdot 3^{1/q}}~~(\mbox{by~Lemma~\ref{Rangeofg}~and~Lemma~\ref{Distribofg}})
\end{eqnarray*}
Therefore 
$$
\frac{\|\widehat{f_k} \|_q}{\|f_k\|_p}\ge \frac{k^{1-1/p-1/q}}{2\cdot 3^{1/q}}\rightarrow\infty
$$
as $k\rightarrow\infty$, which ends the proof of Case 1.
\end{proof}

\begin{proof}[Proof of Proposition \ref{PropR1'XDis}, Case 2]
If the orders of the elements of $X$ are not bounded, then take a sequence $g_n$ in $X$ such that each $g_n$ has order $m_n$ and $m_n\nearrow\infty$. Define
$$
U_n:=\{\hat{x}\in \hat{X}:~\hat{x}(g_n)=1 )\}
$$ 
Hence we have that for every $\hat{x}\in U_n$ and every $1\le j\le m_n$, 
\begin{eqnarray}\label{ValueofgnjonUn}
\big(g_n^j(\hat{x})\big)^{-1}=\big(\hat{x}(g_n)\big)^{-j}=1
\end{eqnarray}
Define 
$$
f_n:=\sum_{j=1}^{m_n}\mathbb{I}_{g_n^j}
$$
Hence $\|f_n\|_p={m_n}^{1/p}$. On the other hand, for every $\hat{x}\in\hat{X}$,
$$
\widehat{f_n}(\hat{x})=\sum_{j=1}^{m_n}\big(g_n^j(\hat{x})\big)^{-1}
$$
Thus, the $L_q$ norm of $\widehat{f_n}$
\begin{eqnarray*}
\|\widehat{f_n}\|_q&=&\left(\int_{\hat{X}}\Big|\sum_{j=1}^{m_n}\left(g_n^j(\hat{x})\right)^{-1}\Big|^q\hat{\alpha}(d\hat{x})\right)^{1/q}
\ge \left(\int_{U_n}\Big|\sum_{j=1}^{m_n}\left(g_n^j(\hat{x})\right)^{-1}\Big|^q\hat{\alpha}(d\hat{x})\right)^{1/q}\\
&= &m_n\hat{\alpha}(U_n)^{1/q}~~(\mbox{by~(\ref{ValueofgnjonUn})})\\
&=&(m_n)^{1-1/q}~~(\mbox{by~Lemma~\ref{Rangeofg}~and~Lemma~\ref{Distribofg}})
\end{eqnarray*}
So we have
$$
\frac{\|\widehat{f_n}\|_q}{\|f_n\|_p}\ge (m_n)^{1-1/p-1/q}\rightarrow\infty
$$
as $n\rightarrow\infty$, which ends the proof of Case 2. 
\end{proof}

\begin{proof} [Proof of Proposition \ref{PropR1'XDis}, Case 3]
Suppose that the order of all elements in ${X}$ are uniformly bounded, we treat $X$ as $\hat{\hat{X}}$, then all lemmas, \ref{Rangeofg} through \ref{DualRestrictionInt} are valid. Define $r$, $\{g_k\}$, $\{M_n\}$, $\{H_n\}$ the same as in the proof of Proposition \ref{PropR2XCom} Case 3 except for substituting $\hat{X}$ by $X$. Define 
$$
f_n:=\mathbb{I}_{M_n}
$$
Then we have $\|f_n\|_p=(r^n)^{1/p}$ by Lemma \ref{LemCardnltOfMn}. On the other hand, for every $\hat{x}\in \hat{X}$, 
$$
\widehat{f_n}(\hat{x})=\sum_{g\in M_n}\hat{x}(-g)
$$
Thus we have
\begin{eqnarray}\label{FTbddorderineq}
\|\widehat{f_n}\|_q =\left(\int_{\hat{x}\in \hat{X}}\Big|\sum_{g\in M_n}\hat{x}(-g)\Big|^q\mathbb{P}(d\hat{x})\right)^{1/q}
\ge \left(\int_{\hat{x}\in H_n}\Big|\sum_{g\in M_n}\hat{x}(-g)\Big|^q\mathbb{P}(d\hat{x})\right)^{1/q}~~~
\end{eqnarray}
Since for every $g\in M_n$, $g(H_n)$ is trivial, thus by Lemma \ref{IndepEqOrd}, Lemma \ref{LemCardnltOfMn} and (\ref{FTbddorderineq}),
$$
\|\widehat{f_n}\|_q \ge \left(\mathbb{P}(H_n)\right)^{1/q}r^n= (r^n)^{1-1/q}
$$
which yields
$$
\frac{\|\widehat{f_n}\|_q}{\|f_n\|_p}\ge (r^n)^{1-1/p-1/q}\rightarrow\infty
$$
as $n\rightarrow\infty$, which ends the proof of Case 3 and provides Proposition \ref{PropR1'XDis}.
\end{proof}

\begin{prop}\label{PropR3'XDis}
If $(1/p,1/q)\in R_3'$, then $C_{p,q}=\infty$.
\end{prop}

We will actually prove a stronger conclusion: $C_{p,q}=\infty$ for all $p>2$ and $q>0$. We still consider the three cases provided in the proof of Proposition \ref{PropR1'XDis}, and assume that the Haar measure on $\hat{X}$ is a probability measure. 

\begin{proof} [Proof of \ref{PropR3'XDis}, Case 1]
Assume that $g$ is the element in $X$ with order infinity. We treat $g$ as an element in $\hat{\hat{X}}$, then Lemma \ref{Rangeofg} and Lemma \ref{Distribofg} are valid by substituting $X$ by $\hat{X}$. Define for every $\hat{x}\in \hat{X}$, 
\begin{eqnarray}\label{repofg}
g(\hat{x}):=\exp\{ib(\hat{x})\}
\end{eqnarray}
where $b(\hat{x}):\hat{X}\rightarrow [0.2\pi]$ is a continuous random variable with uniform distribution by Lemma \ref{Rangeofg} and Lemma \ref{Distribofg}. Define a sequence of functions on $X$:
$$
f_n:=\sum_{k=1}^n\frac{1}{\sqrt{k}} \mathbb{I}_{g^{-2^k}}
$$
Thus $\|f_n\|_p$ are uniformly bounded by the fact that $\sum (1/n)^{p/2}<\infty$ for $p>2$ and $\|\widehat{f_n}\|_2=\|f_n\|_2\rightarrow\infty$. We have, by noting that $g^{2^k}(\hat{x}):=\hat{x}(g^{2^k})$,
\begin{eqnarray*}
\widehat{f_n}(\hat{x})=\sum_{k=1}^n\frac{1}{\sqrt{k}}g^{2^k}(\hat{x})=\sum_{k=1}^n\frac{1}{\sqrt{k}}e^{i2^kb(\hat{x})}
\end{eqnarray*}
Consider the real part:
\begin{eqnarray*}
Re\left(\widehat{f_n}(\hat{x})\right)=\sum_{k=1}^n\frac{1}{\sqrt{k}}\cos\left(2^kb(\hat{x})\right)=:P_n\left(b(\hat{x})\right)
\end{eqnarray*}
where $P_n(y):=\sum_{k=1}^n\frac{1}{\sqrt{k}}\cos\left({2^ky}\right)$ is a lacunary series with $\Lambda=2$ and $P_n$, $A_n$ satisfies condition (\ref{eq:ExtraLacunaryCondition}) in Theorem \ref{Thm:DependCLT}. Thus, apply Theorem \ref{Thm:DependCLT} by setting $E:=(0,2\pi)$, we have that
\begin{eqnarray}\label{eq:LacunaryCLTConv}
\frac{\lambda\left\{y\in (0,2\pi):~P_n(y)/A_n\ge 1\right\}}{2\pi}\rightarrow \frac{1}{2\pi} \int_y^{\infty}e^{-x^2/2}dx>\frac{1}{2\sqrt{2\pi}}e^{-1/2}
\end{eqnarray}
Therefore by Lemma \ref{Rangeofg} and \ref{Distribofg}, we have that the probability measure $\lambda'$ induced by $g$ is $\lambda'=\lambda/2\pi$, therefore we have
\begin{eqnarray*}
\mathbb{P}\left(\hat{x}\in \hat{X}:~b(\hat{x})\in (0,2\pi),~P_n\left(b(\hat{x})\right)\ge A_n\right)&=&\lambda'\left\{y\in (0,2\pi):~P_n(y)\ge A_n\right\}\\
&\rightarrow & \frac{1}{2\pi} \int_y^{\infty}e^{-x^2/2}dx>\frac{1}{2\sqrt{2\pi}}e^{-1/2}
\end{eqnarray*}
This means that there exists $N$ large enough such that for any $n\ge N$, 
\begin{eqnarray}\label{eq:LacunarySeriesCLTIneq}
\mathbb{P}\left(\hat{x}\in \hat{X}:~|\widehat{f_n}(\hat{x})|\ge A_n\right)\ge \mathbb{P}\left(\hat{x}\in \hat{X}:~P_n\left(b(\hat{x})\right)\ge A_n\right)\ge \frac{1}{2\sqrt{2\pi}}e^{-1/2}>0
\end{eqnarray}
Thus by the fact that $A_n\rightarrow\infty$, the inequality (\ref{eq:LacunarySeriesCLTIneq}) means that we always have a subset in $\hat{X}$ with measure at least $\frac{1}{2\sqrt{2\pi}}e^{-1/2}$, such that the value of $|\widehat{f_n}|$ on this set is arbitrarily large as $n\rightarrow\infty$. So we have $\|\widehat{f_n}\|_q\rightarrow\infty$, which ends the proof for Case 1.
\end{proof}

\begin{proof} [Proof of Proposition \ref{PropR3'XDis}, Case 2]
For this case, Lemma \ref{Rangeofg} and Lemma \ref{Distribofg} are valid. Set a sequence $g_n$ in $X$ such that each $g_n$ has order $m_n$ with $m_n\nearrow\infty$. So it is easy to verify that the distribution function of $g_n$ on $[0,2\pi]$ convergence uniformly to $F(x)=x/2\pi$, which means that $g_n\xrightarrow{d} U$ where $U$ is a random variable from some probability space to $[0,2\pi]$ with uniform distribution $\mathbb{P}_U$. We firstly claim that for
\begin{eqnarray}\label{eq:DefofPnTrignmkPoly}
P_n(y)&:=&\sum_{k=1}^n\frac{1}{\sqrt{k}}\cos \left(2^ky\right)\\
A_n&:=&\left(\frac{1}{2}\sum_{k=1}^n\frac{1}{k}\right)^{1/2}
\end{eqnarray}
we have that
\begin{eqnarray}\label{eq:LacunaryCLTUnifDistrib}
\mathbb{P}_U(P_n(U)\ge A_n)\rightarrow\frac{1}{2\pi} \int_y^{\infty}e^{-x^2/2}dx>\frac{1}{2\sqrt{2\pi}}e^{-1/2}
\end{eqnarray}
In fact, it is easy to verify that $P_n$ is a lacunary series with $\Lambda=2$ in Definition \ref{LacunarySeris} and $A_n$'s satisfy condition (\ref{eq:ExtraLacunaryCondition}) in Theorem \ref{Thm:DependCLT}. Thus we have (\ref{eq:LacunaryCLTUnifDistrib}) holds, which means that
\begin{eqnarray}\label{eq:CoclusionLacunaryCLTUnifDistrib}
\mathbb{P}_U(P_{n_0}(U)\ge A_{n_0})\ge \frac{1}{2\sqrt{2\pi}}e^{-1/2}~\mbox{for~all~}n_0~\mbox{large~enough}
\end{eqnarray}
Fix this $n_0$. On the other hand we set $g_n(\hat{x}):=e^{ib_n(\hat{x})}$ with $b_n:\hat{X}\rightarrow[0,2\pi)$ a continuous random variable. Thus $b_n\xrightarrow{d} U$ as $n\rightarrow\infty$. Define, for $n$ large enough such that $m_n>2^{n_0}$, a sequence of functions on $X$:
$$
f_{n,n_0}:=\sum_{k=1}^{n_0}\frac{1}{\sqrt{k}}\mathbb{I}_{g_n^{-2^k}}
$$
thus for $p>2$, $\|f_{n,n_0}\|_p$ is uniformly bounded and
\ben
\widehat{f_{n,n_0}}(\hat{x})=\sum_{k=1}^{n_0}\frac{1}{\sqrt{k}} g_n^{2^k}(\hat{x})=\sum_{k=1}^{n_0}\frac{1}{\sqrt{k}}e^{2^k ib_n(\hat{x})}
\een
Consider the real part of $\widehat{f_{n,n_0}}(\hat{x})$:
\be\label{eq:RealPtFTFunctionSeqFixn0}
Re\left(\widehat{f_{n,n_0}}(\hat{x})\right)=\sum_{k=1}^{n_0}\frac{1}{\sqrt{k}} \cos \left(2^k b_n(\hat{x})\right)=P_{n_0}\left(b_n(\hat{x})\right)
\ee
where $P_{n_0}$ is defined as in (\ref{eq:DefofPnTrignmkPoly}).
Therefore for fixed $n_0$, by the fact that $P_{n_0}$ is continuous, $P_{n_0}\left(b_n(\hat{x})\right)\xrightarrow{d}P_{n_0}(U)$ as $n\rightarrow\infty$. Therefore
\be
\mathbb{P}\left(P_{n_0}\left(b_n(\hat{x})\right)\ge A_{n_0}\right)\rightarrow \mathbb{P}_U(P_{n_0}(U)\ge A_{n_0})\ge \frac{1}{2\sqrt{2\pi}}e^{-1/2}
\ee
as $n\rightarrow\infty$. This means for any $n_0$ large enough, we have an $n(n_0)$ large enough such that for all $n\ge n(n_0)$,
\be
\mathbb{P}\left(P_{n_0}\left(b_n(\hat{x})\right)\ge A_{n_0}\right)\ge \frac{1}{4\sqrt{2\pi}}e^{-1/2}>0
\ee
This means for any $n_0$ large enough, we have an $n(n_0)\in \mathbb{Z}^+$ large enough, such that
\be\label{eq:discretetocts}
\mathbb{P}\left(|\widehat{f_{n(n_0),n_0}}|\ge A_{n_0}\right)\ge  \mathbb{P}\left(P_{n_0}\left(b_{n(n_0)}(\hat{x})\right)\ge A_{n_0}\right)\ge \frac{1}{4\sqrt{2\pi}}e^{-1/2}>0
\ee
Therefore by (\ref{eq:discretetocts}), for any $q>0$,
\begin{eqnarray*}
\|\widehat{f_{n(n_0),n_0}}\|_q\ge \left(\frac{(A_{n_0})^q}{4\sqrt{2\pi}}e^{-1/2} \right)^{1/q}
\end{eqnarray*}
which can be arbitrarily large by the fact that $A_{n_0}\rightarrow\infty$ as $n_0\rightarrow\infty$, which ends the proof of Case 2.
\end{proof}

\begin{rmk}
The method we used for the proof of Case 1 and Case 2 is valid for all $q>0$. In particular, for $q\ge 1$, we have an even better method. The crucial idea is to use the equivalence of $L^p$ ($p\ge 1$) norms of lacunary Fourier sequence mentioned in \cite{Gra14a:book}, p. 240, Theorem 3.7.4, which says that for $q\ge 1$, the $L^q$ norms of lacunary Fourier series with $\Lambda$ defined in Definition \ref{LacunarySeris} are equivalent, where the bounds between the $L^q$ norms only depend on $q$ and $\Lambda$. For Case 1, we define a sequence of functions on $X$:
$$
f_n:=\sum_{k=1}^n\frac{1}{\sqrt{k}} \mathbb{I}_{g^{-2^k}}
$$
Thus $\|f_n\|_p$ are uniformly bounded by the fact that $\sum (1/n)^{p/2}<\infty$ for $p>2$ and $\|\widehat{f_n}\|_2=\|f_n\|_2\rightarrow\infty$. We have, by noting that $g^{2^k}(\hat{x}):=\hat{x}(g^{2^k})$,
\begin{eqnarray*}
\widehat{f_n}(\hat{x})=\sum_{k=1}^n\frac{1}{\sqrt{k}}g^{2^k}(\hat{x})=\sum_{k=1}^n\frac{1}{\sqrt{k}}e^{i2^kb(\hat{x})}=:P_n(b(\hat{x}))
\end{eqnarray*}
where $P_n(y):=\sum_{k=1}^n\frac{1}{\sqrt{k}}e^{i2^ky}$. Then it is easy to check that $P_n$ is a lacunary Fourier series with $L^2$ norms tends to infinity. On the other hand, We have, by Lemma \ref{Rangeofg}, Lemma \ref{Distribofg} and Theorem 3.7.4 in \cite{Gra14a:book}, p. 240, there exists a constant $C_2(2)$ such that 
\begin{eqnarray*}
&&\|\widehat{f_n}\|_1=\mathbb{E}\left(\Big|\sum_{k=1}^n\frac{1}{\sqrt{k}}g^{2^k}\Big|\right)=\mathbb{E}\left(\big|P_n(b)\big|\right)=\|P_n\|_{L^1(\mathbb{T})}\\
&\ge & \left(C_2(2)\right)^{-1}\|P_n\|_{L^2(\mathbb{T})} =\left(C_2(2)\right)^{-1}\left(\mathbb{E}\left(\big|P_n(b)\big|^2\right)\right)^{1/2}=\left(C_2(2)\right)^{-1}\|\widehat{f_n}\|_2\rightarrow\infty 
\end{eqnarray*}
as $n\rightarrow\infty$, moreover by H\"{o}lder's inequality, we have that $\|\widehat{f_n}\|_q\rightarrow\infty$ for all $q\ge 1$, which provides the proof for Case 1.\\

For Case 2, we take a sequence $g_n$ in $X$ such that each $g_n$ has order $m_n$ with $m_n\nearrow\infty$. Define a sequence of function function for $n$ large enough such that $m_n>2^N$
\be\label{FunctionSeqFixN}
f_{n,N}:=\sum_{k=1}^N\frac{1}{\sqrt{k}} \mathbb{I}_{g_n^{-2^k}}
\ee
for some fixed $N\in\mathbb{Z}^+$ and
\be\label{FTFunctionSeqFixN}
\widehat{f_{n,N}}(\hat{x})=\sum_{k=1}^N\frac{1}{\sqrt{k}} g_n^{2^k}(\hat{x})=:P_N(g_n)
\ee
where
$$
P_N(y):=\sum_{k=1}^N\frac{1}{\sqrt{k}} y^{2^k}
$$
Therefore $\|f_{n,N}\|_p$ for $p>2$ is uniformly bounded (independent of $n$ and $N$) and $\|f_{n,N}\|_2\sim \log N$. Note that it is easy to verify that the distribution functions of $g_n$ on $[0, 2\pi]$ converge uniformly to the distribution function of the uniform distribution $U$, which means that $g_n$ converges to $U$ in distribution. Therefore for fixed $N$ in (\ref{FunctionSeqFixN}) and (\ref{FTFunctionSeqFixN}), the Fourier transform $P_N(g_n)\xrightarrow{d} P_N(U)$ by the fact that $P_N$ is continuous. Furthermore for fixed $N$, $|P_N(g_n)|$ and $|P_N(U)|$ are uniformly bounded by $\sum_{k=1}^N 1/\sqrt{k}$. So we have that $\mathbb{E}\left(|P_N(g_n)|\right)\rightarrow \mathbb{E} \left(|P_N(U)|\right)$ 
as $n\rightarrow\infty$. On the other hand, we have that $P_N(e^{2\pi ix})$ is a lacunary Fourier series on $\mathbb{T}$, so we have, by Theorem 3.7.4 in \cite{Gra14a:book}, p. 240
\begin{eqnarray*}
&&\|\widehat{f_{n,N}}\|_1=\mathbb{E}\left(|P_N(g_n)|\right)\rightarrow \mathbb{E} \left(|P_N(U)|\right)=\|P_N\|_{L^1(\mathbb{T})}\\
&\ge & \left(C_2(2)\right)^{-1}\|P_N\|_{L^2(\mathbb{T})}=\left(C_2(2)\right)^{-1}\sum_{k=1}^N\frac{1}{k}\sim\log N
\end{eqnarray*}
This means that for every $N$, $\|\widehat{f_{n,N}}\|_1$ and hence $\|\widehat{f_{n,N}}\|_q$ for $q\ge 1$ is larger than order $\log N$ for $n$ large enough. But the $\|f_{n,N}\|_p$ for $p>2$ is uniformly bounded, which provides the proof of Case 2.
\end{rmk}

\begin{proof} [Proof of Proposition \ref{PropR3'XDis}, Case 3]
Suppose the orders of all elements in $X$ are uniformly bounded, then all lemmas, \ref{Rangeofg} through \ref{DualRestrictionInt} are valid. We define $r$, $\{g_k\}$ and $H_n$ the same as in the proof of Proposition \ref{PropR1'XDis} Case 3 except for substituting $\hat{X}$ by $X$. Then by Lemma \ref{IndepEqOrd}, $g_k$'s are mutually independent and identically distributed. Define a sequence of functions on $X$:
\begin{eqnarray}
f_n:=\sum_{k=1}^n\frac{1}{\sqrt{k}}\mathbb{I}_{g_k^{-1}}
\end{eqnarray}
Thus $\|f_n\|_p$ are uniformly bounded by $\sum (1/n)^{p/2}<\infty$ for $p>2$ and
\begin{eqnarray}\label{FToffnNotSum}
\widehat{f_n}(\hat{x})=\sum_{k=1}^n\frac{1}{\sqrt{k}}g_k(\hat{x})
\end{eqnarray}
by noting that $g_k(\hat{x}):=\hat{x}(g_k)$. We will use Lyapunov Central Limit Theorem \ref{LyapCLT} to prove that there exists a function, $h(n)$, of $n$ with $h(n)\nearrow\infty$ as $n\rightarrow\infty$, such that for every $n$ large enough, the probability 
\begin{eqnarray}\label{Target}
\mathbb{P}\big(Re(\widehat{f_n})\ge h(n)\big)\ge \frac{1}{2\sqrt{2\pi}}e^{-1/2} 
\end{eqnarray}
This means that for every $n$ large enough, we have a measurable set in $\hat{X}$ with a fixed positive probability such that the real part of the partial sum defined by the right hand side of (\ref{FToffnNotSum}) is greater than $h(n)\rightarrow\infty$ on this set. Therefore $\|\widehat{f_n}\|_q\rightarrow\infty$ for any $q\ge 1$, which provides the proof of Case 3. Define $X_k:=Re(g_k)$, then $X_k$'s are i.i.d. by Lemma \ref{IndepEqOrd} with $\mathbb{E}(X_k)=0$ by the fact that $\mathbb{E}(g_k)=0$. Denote $\sigma^2:=\mbox{Var}(X_k)>0$ by the fact that $X_k\neq 0$. Define $h(n)$ by
\begin{eqnarray}
\big(h(n)\big)^2:=\sum_{k=1}^n\mbox{Var}(\frac{X_k}{\sqrt{k}})=\sigma^2\sum_{k=1}^n\frac{1}{k}\sim\sigma^2 \log n
\end{eqnarray}
Therefore we have that $h(n)\nearrow\infty$ as $n\rightarrow\infty$. Define $Y_k:=X_k/\sqrt{k}$, hence $\mathbb{E}(Y_k)=0$ and $\mbox{Var}(Y_k)=\sigma^2/k$ and
\begin{eqnarray}
Re\big(\widehat{f_n}\big)=\sum_{k=1}^nY_k
\end{eqnarray}
We will verify Lyapunov condition (\ref{LyapCdt}). To see this, substitute $\mu_k=0$, $s_n=h(n)$ and $\sigma_k^2=\sigma^2/k$, denote $\sigma':=\mathbb{E}(|X_k|^{2+\delta})$ for some $\delta>0$, therefore $0<\sigma'<\infty$ by the fact that $X_k$ is bounded, we have
\begin{eqnarray*}
&&\frac{1}{s_n^{2+\delta}}\sum_{k=1}^n\mathbb{E}(|Y_k-\mu_k|)^{2+\delta}=\frac{1}{\big(h(n)\big)^{2+\delta}}\sum_{k=1}^n\mathbb{E}\Big(\Big|\frac{X_k}{\sqrt{k}}\Big|^{2+\delta}\Big)\\
&=&\frac{\sigma'}{\big(h(n)\big)^{2+\delta}}\sum_{k=1}^n\frac{1}{k^{1+\delta/2}}\rightarrow 0~\mbox{as}~n\rightarrow\infty
\end{eqnarray*}
by the fact that $h(n)\nearrow\infty$ and the fact that $\sum_{k=1}^n\frac{1}{k^{1+\delta/2}}$ is summable for any $\delta>0$, which provides (\ref{LyapCdt}). Thus, by applying Lyapunov CLT (\ref{LyapConcl}) we have
\begin{eqnarray*}
&&\mathbb{P}\big(Re(\widehat{f_n})\ge h(n)\big)=\mathbb{P}\Big(\sum_{k=1}^n Y_k\ge h(n)\Big)\\
&=& \mathbb{P}\Big(\frac{1}{h(n)}\sum_{k=1}^n Y_k\ge 1 \Big)\rightarrow \mathbb{P}\big(\mathcal{N}(0,1)\ge 1\big)~~~(\mbox{by}~(\ref{LyapConcl}))\\
&>& \frac{1}{2\sqrt{2\pi}}e^{-1/2}
\end{eqnarray*}
where the last inequality is by the lower bound for the normal distribution function:
$$
\mathbb{P}\big(\mathcal{N}(0,1)> t\big)>\frac{1}{\sqrt{2\pi}}\frac{t}{t^2+1}e^{-t^2/2}
$$
which provides (\ref{Target}) and the proof of Case 3 and Proposition \ref{PropR3'XDis}. Proposition \ref{PropR2'XDis}, \ref{PropR1'XDis} and \ref{PropR3'XDis} together provide Theorem \ref{ThmXDis}.
\end{proof}

\section{Implications for uncertainty principles}
\label{sec:up}


Suppose $X$ an LCA group associated with a Haar measure $\alpha$. A {\it probability density function} $f$ on $X$ is an non-negative measurable function on $X$ with
\be\label{DefofDensity}
\int_{x\in X}f(x)\alpha(dx)=1 
\ee

Suppose we have a probability density function $f$ on an LCA group $X$ associated with a Haar measure $\alpha$. 
For $p\in (0,1)\cup (1,\infty)$, define the {\it R\'{e}nyi entropy of order $p$} of $f$ by
\begin{eqnarray}\label{DefRenyiEntrp}
h_{p}(f):=\frac{1}{1-p}\log \left(\int_{x\in X}|f(x)|^p \alpha(dx)\right) .
\end{eqnarray}

We have the following direct corollaries from Theorem \ref{ThmXCom} and Theorem \ref{ThmXDis}:

\begin{cor}\label{WeightedRenyiEntrp}
Suppose $X$ is a compact or a discrete LCA group associated with a Haar measure $\alpha$, and $\hat{X}$ the dual group associated with the dual Haar measure $\hat{\alpha}$ normalized as in Proposition \ref{HarMeas}. Define the following two regions (see \mbox{Figure 4}):
\begin{align*}
U_C&:=\Big\{\Big(\frac{1}{p},\frac{1}{q}\Big)\in [0,\infty)^2:\frac{1}{p}+\frac{1}{q}\le 1, \frac{1}{p}> \frac{1}{2}\Big\},\\
U_D&:=\Big\{\Big(\frac{1}{p},\frac{1}{q}\Big)\in [0,\infty)^2:\frac{1}{p}+\frac{1}{q}\ge 1, \frac{1}{q}< \frac{1}{2}\Big\}.
\end{align*}
For any probability density function $|\psi|^2$ on $X$ (which means $|\hat{\psi}|^2$ is also a probability density function on $\hat{X}$) we have the following weighted R\'{e}nyi entropy uncertainty principle that holds for $(1/p,1/q)\in U_C$ if $X$ is compact and for $(1/p,1/q)\in U_D$ if $X$ is discrete:
\begin{eqnarray}\label{WtdRenyiEntr}
\left(\frac{1}{p}-\frac{1}{2}\right)h_{p/2}(|\psi|^2)+\left(\frac{1}{2}-\frac{1}{q}\right)h_{q/2}(|\hat{\psi}|^2)\ge -\log C_{p,q}
\end{eqnarray}
where $C_{p,q}$ is the norm of the Fourier operator as in Theorem \ref{ThmXCom} and Theorem \ref{ThmXDis}.
\end{cor}
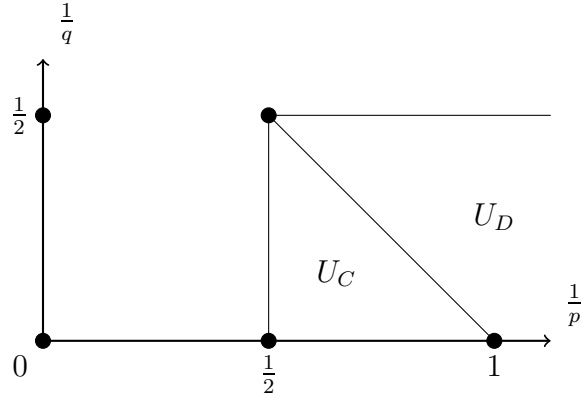
\begin{figure}[H]
	\centering
	\begin{tikzpicture}[scale=1.5,dot/.style={draw,circle,inner sep=2pt,fill},lab/.style={outer sep=2pt,font=\large}]
		\draw [thick,->](0,0)--(0,2.5);
		\draw [thick,->](0,0)--(4.5,0);
		\draw (2,2)--(4,0);
		\draw (2,0)--(2,2);
		\draw (4.5,2)--(2,2);
		\node[dot] at (0,0) {} node[lab] at (0,0) [below left] {$0$};
		\node[dot] at (2,0) {} node[lab] at (2,0) [below] {$\frac{1}{2}$};
		\node[dot] at (4,0) {} node[lab] at (4,0) [below] {$1$} node[lab] at (4.5,0) [above right] {$\frac{1}{p}$};
		\node[dot] at (0,2) {} node[lab] at (0,2) [left] {$\frac{1}{2}$};
		\node[lab] at (0,2.5) [above right] {$\frac{1}{q}$};
		\node[dot] at (2,2) {};
	
		\node[lab] at (2.6,0.6) {$U_C$};
		\node[lab] at (4.0,1.1) {$U_D$};
		
		
		
		
	\end{tikzpicture}
	\caption{Weighted uncertainty principle regions.}
	\label{figure:4}
\end{figure}

\begin{cor}\label{WeightedRenyiEntrpDNE}
Suppose $X$ is a compact or a discrete LCA group associated with a Haar measure $\alpha$, and $\hat{X}$ the dual group 
associated with the dual Haar measure $\hat{\alpha}$ normalized as in Proposition \ref{HarMeas}. 
Suppose also that $X$ is not finite.
Define the following two regions (see Figure 5):
\begin{align*}
U_{CN}&:=\Big\{\Big(\frac{1}{p},\frac{1}{q}\Big)\in [0,\infty)^2:\frac{1}{p}+\frac{1}{q}> 1, \frac{1}{q}\le \frac{1}{2}\Big\},\\
U_{DN}&:=\Big\{\Big(\frac{1}{p},\frac{1}{q}\Big)\in [0,\infty)^2:\frac{1}{p}+\frac{1}{q}< 1, \frac{1}{p}\ge \frac{1}{2}\Big\}.
\end{align*}
For any constant $C<0$ and any pair $(p,q)$ satisfying $(1/p,1/q)\in U_{CN}$ if $X$ is compact and $(1/p,1/q)\in U_{DN}$ if $X$ is discrete, there exists a probability density function $|\psi|^2$ on $X$ (which means $|\hat{\psi}|^2$ is also a probability density function on $\hat{X}$) depending on $C$, $p$ and $q$ such that:
\begin{eqnarray}\label{WtdRenyiEntrDNE}
\left(\frac{1}{p}-\frac{1}{2}\right)h_{p/2}(|\psi|^2)+\left(\frac{1}{2}-\frac{1}{q}\right)h_{q/2}(|\hat{\psi}|^2)<C.
\end{eqnarray}
\end{cor}

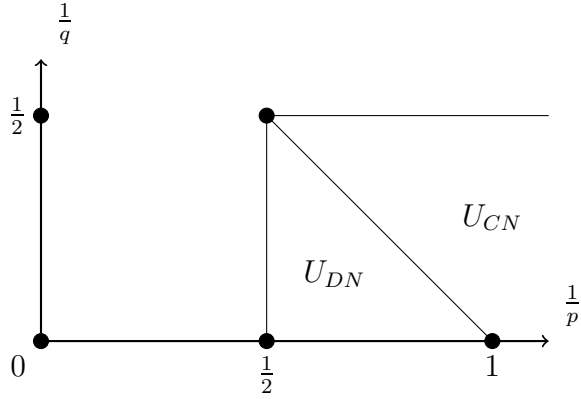
\begin{figure}[H]
	\centering
	\begin{tikzpicture}[scale=1.5,dot/.style={draw,circle,inner sep=2pt,fill},lab/.style={outer sep=2pt,font=\large}]
		\draw [thick,->](0,0)--(0,2.5);
		\draw [thick,->](0,0)--(4.5,0);
		\draw (2,2)--(4,0);
		\draw (2,0)--(2,2);
		\draw (4.5,2)--(2,2);
		\node[dot] at (0,0) {} node[lab] at (0,0) [below left] {$0$};
		\node[dot] at (2,0) {} node[lab] at (2,0) [below] {$\frac{1}{2}$};
		\node[dot] at (4,0) {} node[lab] at (4,0) [below] {$1$} node[lab] at (4.5,0) [above right] {$\frac{1}{p}$};
		\node[dot] at (0,2) {} node[lab] at (0,2) [left] {$\frac{1}{2}$};
		\node[lab] at (0,2.5) [above right] {$\frac{1}{q}$};
		\node[dot] at (2,2) {};
	
		\node[lab] at (2.6,0.6) {$U_{DN}$};
		\node[lab] at (4.0,1.1) {$U_{CN}$};
		
		
		

	\end{tikzpicture}
	\caption{Regions such that weighted uncertainty principle does not exist.}
	\label{figure:2}
\end{figure}

Corollary~\ref{WeightedRenyiEntrpDNE} is particularly interesting because it identifies regions where
a natural weighted uncertainty principle  {\it fails} to hold for infinite groups that are compact or discrete.

Corollaries~\ref{WeightedRenyiEntrp} and \ref{WeightedRenyiEntrpDNE} are
the most that can be extracted directly from our results on the $(p,q)$-norms of the Fourier
transform. However, it is possible to obtain other expressions of the entropic uncertainty principle
by exploiting the monotonicity property of the R\'enyi entropies. 
Indeed, if $X$ is $\mathbb{R}^n$, $\mathbb{T}^n$, $\mathbb{Z}^n$ or $\mathbb{Z}/m\mathbb{Z}$, 
Zozor, Portesi and Vignat \cite{ZPV08} used this idea to obtain unweighted R\'enyi entropic uncertainty 
principle of the form 
\be\label{UwtdEntrpPrspecial}
h_{\alpha}(|\psi|^2)+h_{\beta}(|\hat{\psi}|^2)\ge B_{\alpha,\beta}
\ee
for corresponding domains of $\alpha$ and $\beta$ with the corresponding sharp constants $B_{\alpha,\beta}$. 
We now observe that a similar result holds for  general LCA groups.

\begin{thm}\label{thm:unweighteduncprip}
Let $X$ be an LCA group isomorphic to a direct product of a finite number of copies of $\mathbb{R}$ and an LCA group $B$ which contains an open compact subgroup: $X=\mathbb{R}^n\times B$, and let $X$ be equipped with a Haar measure $\alpha$. Let $\hat{X}$ be the dual group with the Haar measure $\hat{\alpha}$ normalized to be such that the Plancherel identity is valid. Then for any probability density function $|\psi^2|$ on $X$ and any $p\ge 0$, $q\ge 2$ such that $1/p+1/q\ge 1$, we have the R\'enyi entropic uncertainty principle
\begin{align}\label{eq:unweighteduncprip}
h_{p/2}(|\psi|^2)+h_{q/2}(|\hat{\psi}|^2)& \ge n\log C(p,q) ,
\end{align}
where 
\begin{align*}
C(p,q)=\left(\max\left(q^{1/(q-2)}q'^{1/q'-2},p^{1/(p-2)}p'^{1/p'-2}\right)\right),
\end{align*}
where $p'$ and $q'$ are the H\"older duals of $p$ and $q$ respectively. Further if $X$ is a discrete or compact abelian group equipped with Haar measure $\alpha$, let $\hat{X}$ be the dual group with the Haar measure $\hat{\alpha}$ normalized so that the Plancherel identity is valid. Then for any probability density function $|\psi^2|$ on $X$ and any non-negative $p\ge 0$, $q\ge 0$ such that $1/p+1/q\ge 1$, we have the sharp R\'enyi entropic uncertainty principle
\begin{align}\label{eq:unweighteduncpripdc}
h_{p/2}(|\psi|^2)+h_{q/2}(|\hat{\psi}|^2)& \ge 0 .
\end{align}
\end{thm}

\begin{proof}
We first have (\ref{eq:unweighteduncprip}) holds for all $p\ge 0$, $q\ge 2$ with $1/p+1/q=1$ by Hausdorff-Young inequality. Thus For any pair of $p\ge 0$ and $q\ge 2$ with $1/p+1/q> 1$, consider 
$(\tilde{p},\tilde{q}):=(p,p')$ or $(\tilde{p},\tilde{q}):=(q',q)$ and by monotonicity of R\'enyi entropy, we have
\begin{eqnarray*}
h_{p/2}(|\psi|^2)+h_{q/2}(|\hat{\psi}|^2)\ge h_{\tilde{p}/2}(|\psi|^2)+h_{\tilde{q}/2}(|\hat{\psi}|^2)\ge n\log C(p,q)
\end{eqnarray*}
which provides (\ref{eq:unweighteduncprip}). 

Next we prove \eqref{eq:unweighteduncpripdc}. Similarly we have that (\ref{eq:unweighteduncpripdc}) holds for $1/p+1/q=1$ with $q\ge 2$ by Hausdorff-Young inequality. For $q<2$ and $p> 2$, we treat $\hat{\psi}$ a function on $\hat{X}$ and apply Hausdorff-Young inequality again, we have
\begin{eqnarray}\label{eq:unweighteduncpripq<2}
h_{p/2}(|\hat{\hat{\psi}}|^2)+h_{q/2}(|\hat{\psi}|^2)\ge 0
\end{eqnarray}
By the fact that $\hat{\psi}\in L^2(\hat{X})$ hence $\hat{\hat{\psi}}\in L^2(X)$ and $\hat{\hat{\psi}}(x)=\psi(-x)$ by Theorem \ref{thm:fouriertwice}. So by (\ref{eq:unweighteduncpripq<2}), we have
\begin{eqnarray}
h_{p/2}(|\psi|^2)+h_{q/2}(|\hat{\psi}|^2)=h_{p/2}(|\hat{\hat{\psi}}|^2)+h_{q/2}(|\hat{\psi}|^2)\ge 0
\end{eqnarray}
which provides the desired inequality when $1/p+1/q=1$.

For any pair of $(p,q)$ with $1/p+1/q> 1$, consider 
$(\tilde{p},\tilde{q}):=(p,p')$ or $(\tilde{p},\tilde{q}):=(q',q)$ and by monotonicity of R\'enyi entropy, we have
\begin{eqnarray*}
h_{p/2}(|\psi|^2)+h_{q/2}(|\hat{\psi}|^2)\ge h_{\tilde{p}/2}(|\psi|^2)+h_{\tilde{q}/2}(|\hat{\psi}|^2)\ge0
\end{eqnarray*}
which provides (\ref{eq:unweighteduncpripdc}). Moreover, this bound can be attained 
by any normalized frequency basis if $X$ is compact (i.e., functions of the form $\gamma/\sqrt{\alpha(X)}$ with any $\gamma\in\hat{X}$),
and by any time basis if $X$ is discrete (i.e., normalized delta functions on $X$). 
\end{proof}

Not suprisingly given that a key step in the proof was rewriting the Hausdorff-Young inequality  
as an entropy inequality, Theorem \ref{thm:unweighteduncprip} is a direct generalization
of Hirschman's entropic uncertainty principle (as improved by Beckner through his determination
of the sharp constant),
\ben
h(|\psi|^2)+h(|\hat{\psi}|^2)& \ge  n\log(e/2) .
\een

On the other hand, by applying inequality \eqref{eq:unweighteduncpripdc} in Theorem \ref{thm:unweighteduncprip} and letting $p$ and $q$ go to 0, 
we have an uncertainty principle related to the support of a density function.

\begin{cor}\label{cor:Unweightedpripqgoto0}
If $X$ is a compact or discrete abelian group, under the assumption of Theorem \ref{thm:unweighteduncprip}, then for any probability density function $|\psi^2|$ on $X$, we have
\begin{eqnarray}\label{eq:unweightedpq=0}
\alpha\left(\mbox{supp}(\psi)\right)
\hat{\alpha}\left(\mbox{supp}\left(\hat{\psi}\right)\right)\ge 1
\end{eqnarray}
Moreover, if $X$ is compact or discrete, then the equality can be attained by any normalized frequency basis if $X$ is compact and any normalized time basis if $X$ is discrete.
\end{cor}

For finite abelian groups, a series of works have focused on statements involving the cardinalities of supports of a function
and its Fourier transform. Donoho and Stark \cite{DS89} proved a lower bound for the product of support sizes and identified
extremals for cyclic groups and \cite{MOP04} gave a simple proof of the extension for finite abelian groups (although 
the same fact for general LCA groups was already obtained by Matolcsi and Sz\H{u}cs \cite{MS73} and Smith \cite{Smi90}).
We observe that if $X$ is a finite abelian group, then Corollary~\ref{cor:Unweightedpripqgoto0} reduces immediately to an uncertainty principle 
of Donoho-Stark-type.

\begin{cor}\label{cor:FiniteUP}
Suppose that $X$ is a finite abelian group with cardinality $N$ and a counting Haar measure $1/\sqrt{N}$. Let $\hat{X}$ be its dual group with Haar measure $1/\sqrt{N}$. Then for any function $\psi$ on $X$, denote $N_t$ and $N_w$ the number of nonzero entries in $\psi$ and in $\hat{\psi}$ (here we use the same notation as in \cite{DS89}), we have
\begin{eqnarray}\label{eq:FiniteUP}
N_t\cdot N_w\ge N
\end{eqnarray}
Moreover, the inequality (\ref{eq:FiniteUP}) is sharp and the equality can be attained by any frequency basis and time basis of $X$.
\end{cor}

\begin{proof}
Note that it suffices to prove the case that $\psi$ is such that $|\psi|^2$ is a probability mass function, 
because normalization does not change the cardinality of the support of $\psi$ and $\hat{\psi}$. By (\ref{eq:unweightedpq=0}), we have that
\ben
\alpha\left(\mbox{supp}(\psi)\right)
\hat{\alpha}\left(\mbox{supp}\left(\hat{\psi}\right)\right)\ge 1\Rightarrow \frac{N_t}{\sqrt{N}}\cdot\frac{N_w}{\sqrt{N}}\ge 1\Rightarrow N_tN_w\ge N
\een
which provides (\ref{eq:FiniteUP}). 
\end{proof}

A trivial application of the AM-GM inequality implies that
\ben
N_t+N_w\ge 2\sqrt{N_tN_w}\ge 2\sqrt{N} ,
\een
constraining the sum of the support sizes of a function and its Fourier transform for a finite abelian group. As one might
expect, however, this is not sharp. A sharp inequality for the sum of supports was found for cyclic groups of prime order
by Tao \cite{Tao05} (the correct bound turns out to be $N+1$ rather than $2\sqrt{N}$), 
with a generalization to finite cyclic groups obtained by Ram Murty and Whang \cite{RW12}.


It is worth noting that Dembo, Cover and Thomas \cite{DCT91} obtained a Hirschman-type inequality for finite cyclic groups (i.e.,
a Shannon entropy version of inequality (\ref{eq:FiniteUP}), which is stronger than the inequality on product of support sizes), 
and that Przebinda, DeBrunner and \"Ozayd{\i}n \cite{PDO01} obtained a characterization of extremals for this inequality.

\section*{Acknowledgements} 

The authors are grateful to Philippe Jaming for useful comments on an earlier draft of this paper.


\end{document}